\documentclass[11pt]{article}
\usepackage[centertags]{amsmath}
\usepackage{amsfonts}
\usepackage{amssymb}
\usepackage{amsthm}
\usepackage{newlfont}

\newtheorem{Theorem}{Theorem}[section]
\newtheorem{Definition}[Theorem]{Definition}
\newtheorem{Proposition}[Theorem]{Proposition}
\newtheorem{Lemma}[Theorem]{Lemma}
\newtheorem{Corollary}[Theorem]{Corollary}
\newtheorem{Remark}[Theorem]{Remark}

\newtheorem{Hypothesis}{Hypothesis}

\numberwithin{equation}{section}

\begin{document}

\def\le{\left}
\def\r{\right}
\def\cost{\mbox{const}}
\def\a{\alpha}
\def\d{\delta}
\def\ph{\varphi}
\def\e{\epsilon}
\def\la{\lambda}
\def\si{\sigma}
\def\La{\Lambda}
\def\B{{\cal B}}
\def\A{{\mathcal A}}
\def\L{{\mathcal L}}
\def\O{{\mathcal O}}
\def\bO{\overline{{\mathcal O}}}
\def\F{{\mathcal F}}
\def\K{{\mathcal K}}
\def\H{{\mathcal H}}
\def\D{{\mathcal D}}
\def\C{{\mathcal C}}
\def\M{{\mathcal M}}
\def\N{{\mathcal N}}
\def\G{{\mathcal G}}
\def\T{{\mathcal T}}
\def\R{{\mathcal R}}
\def\I{{\mathcal I}}

\def\bw{\overline{W}}
\def\phin{\|\varphi\|_{0}}
\def\s0t{\sup_{t \in [0,T]}}
\def\lt{\lim_{t\rightarrow 0}}
\def\iot{\int_{0}^{t}}
\def\ioi{\int_0^{+\infty}}
\def\ds{\displaystyle}
\def\pag{\vfill\eject}
\def\fine{\par\vfill\supereject\end}
\def\acapo{\hfill\break}

\def\beq{\begin{equation}}
\def\eeq{\end{equation}}
\def\barr{\begin{array}}
\def\earr{\end{array}}
\def\vs{\vspace{.1mm}   \\}
\def\rd{\reals\,^{d}}
\def\rn{\reals\,^{n}}
\def\rr{\reals\,^{r}}
\def\bD{\overline{{\mathcal D}}}
\newcommand{\dimo}{\hfill \break {\bf Proof - }}
\newcommand{\nat}{\mathbb N}
\newcommand{\E}{\mathbb E}
\newcommand{\Pro}{\mathbb P}
\newcommand{\com}{{\scriptstyle \circ}}
\newcommand{\reals}{\mathbb R}

\title{A Khasminskii type averaging principle for stochastic reaction-diffusion equations }

\author{Sandra Cerrai\\
Dip. di Matematica per le Decisioni\\
Universit\`{a} di Firenze\\
Via C. Lombroso 6/17\\
 I-50134 Firenze, Italy
}


\maketitle

\begin{abstract}
We prove that an averaging principle holds for a general class of
stochastic reaction-diffusion systems, having unbounded
multiplicative noise, in any space dimension. We show that the
classical Khasminskii approach for systems with a finite number of
degrees of freedom can be extended to infinite dimensional
systems.
\end{abstract}

\section{Introduction}

Consider the deterministic system with a finite number of degrees
of freedom
\begin{equation}
\label{odeper}\le\{\begin{array}{l}
\ds{\frac{d\hat{X}_\e}{dt}(t)=\e\, b(\hat{X}_\e(t),\hat{Y}_\e(t)),\ \ \ \ \ \hat{X}_\e(0)=x \in\,\mathbb{R}^n,}\\
\vs \ds{\frac{d\hat{Y}_\e}{dt}(t)=g(\hat{X}_\e(t),\hat{Y}_\e(t)),\
\ \ \ \ \hat{Y}_\e(0)=y \in\,\mathbb{R}^k,}
\end{array} \r.
\end{equation}
for some parameter $0<\e<<1$ and some mappings
$b:\mathbb{R}^n\times \mathbb{R}^k\to \mathbb{R}^n$ and
$g:\mathbb{R}^n\times \mathbb{R}^k\to \mathbb{R}^k$. Under
reasonable conditions on $b$ and $g$, it is clear that, as the
parameter $\e$ goes to zero, the first component $\hat{X}_\e(t)$
of the perturbed system \eqref{odeper} converges to the constant
first component $x$ of the unperturbed system, uniformly with
respect to $t$ in a bounded interval $[0,T]$, for any fixed $T>0$.

But in applications what is more interesting is the behavior of
$\hat{X}_\e(t)$ for $t$ in intervals of order $\e^{-1}$ or even
larger. Actually, it is indeed on those time scales that the most
significant changes happen, such as exit from the neighborhood of
an equilibrium point or of a periodic trajectory. With the natural
time scaling $t\mapsto t/\e$, if we set
$X_\e(t):=\hat{X}_\e(t/\e)$ and $Y_\e(t):=\hat{Y}_\e(t/\e)$,
equation \eqref{odeper} can be rewritten as
\begin{equation}
\label{odeper2}\le\{\begin{array}{l}
\ds{\frac{dX_\e}{dt}(t)=b(X_\e(t),Y_\e(t)),\ \ \ \ \ X_\e(0)=x \in\,\mathbb{R}^n,}\\
\vs \ds{\frac{dY_\e}{dt}(t)=\frac 1\e\,g(X_\e(t),Y_\e(t)),\ \ \ \
\ Y_\e(0)=y \in\,\mathbb{R}^k,}
\end{array} \r.\end{equation} and with this time scale the variable $X_\e$ is always referred as
the {\em slow} component and $Y_\e$ as the {\em fast} component.
In particular, the study of system \eqref{odeper} in time
intervals of order $\e^{-1}$ is equivalent to the study of system
\eqref{odeper2} on finite time intervals.

Now, assume that for any $x \in\,\mathbb{R}^n$ there exists the
limit
\begin{equation} \label{ipotesi}
\bar{b}(x)=\lim_{T\to \infty}\frac 1T\int_0^T b(x,Y^x(t))\,dt,
\end{equation}
where $Y^x(t)$ is the fast motion with frozen slow component $x
\in\,\mathbb{R}^n$
\[\frac{d Y^x}{dt}(t)=g(x,Y^x(t)),\ \ \ \ \
Y^x(0)=y.\]
Such a limit exists for example in the case the
function $Y^x(t)$ is periodic.
 Moreover,  assume that the mapping
$\bar{b}:\mathbb{R}^n\to \mathbb{R}^n$ satisfies some reasonable
assumption, for example it is Lipschitz continuous. In this
setting, the {\em averaging principle} says that the trajectory of
$X_\e$ can be approximated by the solution $\bar{X}$ of the
so-called {\em averaged} equation
\[\frac{d\bar{X}}{dt}(t)=\bar{b}(\bar{X}(t)),\ \ \ \ \ \ \bar{X}(0)=x,\]
uniformly in $t \in\,[0,T]$, for any fixed $T>0$. This means that
by averaging principle  a good approximation of the slow motion
can be obtained by averaging its parameters in the fast variables.

\medskip

The theory of averaging, originated by Laplace and Lagrange, has
been applied in its long history in many fields, as for example
celestial mechanics, oscillation theory and radiophysics, and for
a long period it has been used  without a rigorous mathematical
justification. The first rigorous results are due to Bogoliubov
(cfr. \cite{bogo}) and concern both the case of uncoupled systems
and the case of $g(x,y)=g(x)$. Further developments of the theory,
for more general systems, were obtained by Volosov, Anosov and
Neishtadt (to this purpose we refer to \cite{neistadt} and
\cite{volosov}) and a good understanding of the involved phenomena
was obtained by Arnold et al. (cfr. \cite{akn}).

A further  development in the theory of averaging, which is of
great interest in applications, concerns the case of random
perturbations of dynamical systems. For example,  in system
\eqref{odeper} the coefficient $g$ may be assumed to depend also
on a parameter $\omega \in\,\Omega$, for some probability space
$(\Omega,\mathcal{F},\mathbb{P})$, so that the fast variable is a
random process, or even the perturbing coefficient $b$ may be
taken random. Of course, in these cases one has to reinterpret
condition \eqref{ipotesi} and the type of convergence of the
stochastic process $X_\e$ to $\bar{X}$. One possible way is to
require \eqref{ipotesi} with probability $1$, but in most of the
cases this assumption turns out to be too restrictive. More
reasonable is to have \eqref{ipotesi} either in probability or in
the mean, and in this case one expects to have convergence in
probability of $X_\e$ to $\bar{X}$. As far as averaging for
randomly perturbed systems is concerned, it is worthwhile to quote
the important work of Brin, Freidlin and Wentcell  (see \cite{brfr},
\cite{freidlin}, \cite{frven}, \cite{frven2}) and also the work of
Kifer and Veretennikov (see for example \cite{kif1}, \cite{kif2},
\cite{kif3}, \cite{kif4} and \cite{vere}).

An important contribution in this direction is given by  Khasminskii  with his paper 
\cite{khas} appeared in 1968. In this paper  he  has
considered the following system of stochastic differential
equations
\begin{equation}
\label{kha} \le\{
\begin{array}{l}
\ds{dX^\e(t)=A(X^\e(t),Y^\e(t))\,dt+\sum_{r=1}^l
\si^r(X_i^\e(t),Y^\e(t))\,dw_r(t),\ \ \ \ X^\e(0)=x_0,}\\
\vs \ds{dY^\e(t)=\frac 1\e B(X^\e(t),Y^\e(t))\,dt+\frac
1{\sqrt{\e}}\,\sum_{r=1}^l \varphi^r(X^\e(t),Y^\e(t))\,dw_r(t),\ \
\ \ Y^\e(0)=y_0,}
\end{array}\r.
\end{equation}
for some $l$-dimensional Brownian motion
$w(t)=(w_1(t),\ldots,w_l(t))$. In this case the perturbation in
the slow motion is given by the sum of a deterministic part and a
stochastic part
\[\e\,b(x,y)\,dt=\e A(x,y)\,dt+\sqrt{\e}\,\sigma(x,y)dw(t),\]
and the fast motion is described by a stochastic differential
equation.

In \cite{khas} the coefficients $A:\reals^{l_1}\times
\reals^{l_2}\to \reals^{l_1} $ and $\si:\reals^{l_1}\times
\reals^{l_2}\to M(l\times l_1)$ in the slow motion equation are
assumed to be Lipschitz continuous and uniformly bounded in $y
\in\,\mathbb{R}^{l_2}$. The coefficients $B:\reals^{l_1}\times
\reals^{l_2}\to \reals^{l_2} $ and $\varphi:\reals^{l_1}\times
\reals^{l_2}\to M(l\times l_2)$ in the fast motion equation are
assumed to be Lipschitz continuous, so that in particular the fast
equation with frozen slow component $x$
\[dY^{x,y}(t)=B(x, Y^{x,y}(t))\,dt+\sum_{r=1}^l
\varphi^r(x,Y^{x,y}(t))\,dw_r(t),\ \ \ \ Y^{x,y}(0)=y,\] admits a
unique solution $Y^{x,y}$, for any $x \in\,\mathbb{R}^{l_1}$ and $y
\in\,\mathbb{R}^{l_2}$. Moreover, it is assumed that there exist
two mappings $\bar{A}:\reals^{l_1}\to \reals^{l_1} $ and
$\{a_{ij}\}:\reals^{l_1}\to M(l\times l_1)$ such that
\begin{equation}
\label{kha1} \le|\E\,\frac
1T\int_0^TA(x,Y^{x,y}(t))\,dt-\bar{A}(x)\r|\leq
\a(T)\,\le(1+|x|^2\r),
\end{equation}
and for any $i=1,\ldots,l_1$ and $j=1,\ldots,l_2$
\[\le|\E\,\frac
1T\int_0^T\sum_{r=1}^l \si_i^r
\si_j^r(x,Y^{x,y}(t))\,dt-a_{ij}(x)\r|\leq
\a(T)\,\le(1+|x|^2\r),\] for some function $\a(T)$ vanishing as
$T$ goes to infinity.

In his paper, Khasminskii shows that an averaging principle holds
for system \eqref{kha}. Namely, the slow motion $X^\e(t)$
converges in weak sense, as $\e$ goes to zero, to the solution
$\bar{X}$ of the {\em averaged} equation
\[dX(t)=\bar{A}(X(t))\,dt+\bar{\si}(X(t))\,dw(t),\ \ \ \ \
X(0)=x_0,\] where $\bar{\si}$ is the square root of the matrix
$\{a_{ij}\}$.

\bigskip

The behavior of solutions of infinite dimensional systems   on
time intervals of order $\e^{-1}$ is at present not very well
understood, even if applied mathematicians do believe that the
averaging principle holds and usually  approximate of the slow
motion by the averaged motion, also with $n=\infty$. As far as we
know, the literature on averaging for systems with an infinite
number of degrees of freedom is extremly poor (to this purpose we
refer to the papers \cite{seivr}  by Seidler-Vrko\v{c}  and \cite{massei} by Maslowskii-Seidler-Vrko\v{c}, concerning with
averaging for Hilbert-space valued solutions of stochastic evolution equations depending on a small parameter, and to the paper \cite{kukpia} by Kuksin and Piatnitski 
concerning with averaging for a randomly perturbed KdV equation) and almost all
has still to be done.

In the present paper we are trying to extend the Khasminskii
argument to a system with an infinite number of degrees of
freedom.  We are dealing with the following system of stochastic
reaction-diffusion equations on a bounded domain $D\subset
\reals^d$, with $d\geq 1$,

\begin{equation}
\label{eq0} \le\{
\begin{array}{l}
\ds{\frac{\partial u_\e}{\partial t}(t,\xi)=\mathcal{A}_1
u_\e(t,\xi)+b_1(\xi,u_\e(t,\xi),v_\e(t,\xi))+g_1(\xi,u_\e(t,\xi),v_\e(t,\xi))\,\frac{\partial
w^{Q_1}}{\partial t}(t,\xi),}\\
\vs \ds{\frac{\partial v_\e}{\partial t}(t,\xi)=\frac
1\e\le[\mathcal{A}_2
v_\e(t,\xi)+b_2(\xi,u_\e(t,\xi),v_\e(t,\xi))\r]+\frac 1{\sqrt{\e}}
g_2(\xi,u_\e(t,\xi),v_\e(t,\xi))\,\frac{\partial
w^{Q_2}}{\partial t}(t,\xi), }\\
\vs \ds{u_\e(0,\xi)=x(\xi),\ \ \ \ v_\e(0,\xi)=y(\xi),\ \ \ \ \
\xi
\in\,D,}\\
\vs \ds{ \mathcal{N}_{1} u_\e\,(t,\xi)=\mathcal{N}_{2}
v_\e\,(t,\xi)=0,\ \ \ \ t\geq 0,\ \ \ \ \xi \in\,\partial D.}
\end{array}\r.
\end{equation}
for a positive parameter $\e<<1$. The stochastic perturbations are
given by Gaussian noises which are  white in time and colored in
space, in the case of space dimension $d>1$, with covariances
operators $Q_1$ and $Q_2$. The operators $\mathcal{A}_1$ and
$\mathcal{A}_2$ are second order uniformly elliptic operators,
having continuous coefficients on $D$, and the boundary operators
$\mathcal{N}_1$ and $\mathcal{N}_2$ can be either the identity
operator (Dirichlet boundary condition) or a first order operator
satisfying a uniform nontangentiality condition.

In our previous paper \cite{cf}, written in collaboration with M.
Freidlin, we have considered the simpler case of $g_1\equiv 0$ and
$g_2\equiv 1$, and we have proved that an averaging principle is
satisfied by using  a completely different approach based on
Kolmogorov equations and martingale solutions of stochastic
equations, which is more in the spirit of the general method
introduced by Papanicolaou, Strook and Varadhan  in their paper
\cite{papanicolaou} of 1977. Here, we are considering the case of
general reaction coefficients $b_1$ and $b_2$ and diffusion
coefficients $g_1$ and $g_2$, and the method based on the
martingale approach seems to be very complicated to be applied.

We would like to stress that both here and  in our previous paper
\cite{cf} we are considering averaging for randomly perturbed
reaction-diffusion systems, which are of interest in the
description of diffusive phenomena in reactive media, such as
combustion, epidemic propagation, diffusive transport of chemical
species through cells and dynamics of populations. However the
arguments we are using adapt easily to more general models of
semi-linear stochastic partial differential equations.

Together with system \eqref{eq0}, for any $x, y \in\,H:=L^2(D)$ we introduce the fast motion
equation
\[\le\{
\begin{array}{l}
\ds{\frac{\partial v}{\partial t}(t,\xi)=\le[\mathcal{A}_2
v(t,\xi)+b_2(\xi,x(\xi),v(t,\xi))\r]+g_2(\xi,x(\xi),v(t,\xi))\,\frac{\partial
w^{Q_2}}{\partial t}(t,\xi),}\\
\vs \ds{v(0,\xi)=y(\xi),\ \ \ \xi \in\,D, \ \ \ \ \ \ \
\mathcal{N}_{2} v\,(t,\xi)=0,\ \ \ \ t\geq 0,\ \ \ \ \xi
\in\,\partial D,}
\end{array}\r.\]
with initial datum $y$ and frozen slow component $x$, whose solution is
denoted by  $v^{x,y}(t)$. The previous equation has been widely
studied, as far as existence and uniqueness of solutions are
concerned. In Section 3 we introduce the transition semigroup
$P^x_t$ associated with it and, by using methods and results from
our previous paper \cite{cerrai2}, we study its asymptotic
properties and its dependence on the parameters $x$ and $y$ (cfr.
also 
\cite{tesi} and  \cite{cerrai1}).

Under this respect, in addition to  suitable conditions on the
operators $\mathcal{A}_i$ and $Q_i$ and on the coefficients $b_i$
and $g_i$, for $i=1,2$ (see Section \ref{sec2} for all hypotheses), in the spirit of  Khasminskii's work we assume that
there exist a mapping $\a(T)$ which vanishes as $T$ goes to infinity
and two Lipschitz-continuous mappings $\bar{B}_1:H\to H$ and $\bar{G}:H\to {\cal L}(L^\infty(D),H))$ such that for any choice of 
$T>0$, $t\geq 0$ and $x,y \in\,H$
\begin{equation}
\label{ops2bis}
\E\,\le|\frac
1T\int_t^{t+T}\le<B_1(x,v^{x,y}(s)),h\r>_{H}\,ds-\le<\bar{B}_1(x),h\r>_{H}\r|\leq
\a(T)\,(1+|x|_H+|y|_H)\,|h|_{H},
\end{equation}
for any $h \in\,H$,
and
\begin{equation} \label{ops2}\begin{array}{l}
\ds{\le|\frac 1T\int_t^{t+T}\E\,\le<G_1(x,v^{x,y}(s))h,
G_1(x,v^{x,y}(s))k\r>_H\,ds-\le<\bar{G}(x)h,\bar{G}(x)k\r>_H\r|}\\
\vs \ds{\leq \a(T)\,\le(1+|x|^2_H+|y|^2_H\r)\,|h|_{L^\infty(D)}
|k|_{L^\infty(D)},}
\end{array}\end{equation}
for any $h,k \in\,L^\infty(D)$. Here, $B_1$ and
$G_1$ are the Nemytskii operators associated with $b_1$ and $g_1$,
respectively. Notice that, unlike $B_1$ and $G_1$ which are local
operators, the coefficients $\bar{B}$ and $\bar{G}$ are not local. Actually, 
they are defined as general mappings  on $H$ and, also in applications,  there is no reason why they should be composition operators.
 
In Section \ref{sec3}, we describe some remarkable situations in
which conditions \eqref{ops2bis} and \eqref{ops2} are fulfilled: for example when the fast
motion admits a strongly mixing invariant measure $\mu^x$,
for any fixed frozen slow component $x \in\,H$, and the diffusion
coefficient $g_1$ of the slow motion equation is bounded and
non-degenerate.

\medskip

Our purpose is showing that under the above conditions  the slow motion $u_\e$
converges weakly to the solution $\bar{u}$ of the averaged
equation
\begin{equation}
\label{avintro}
\le\{\begin{array}{l} \ds{\frac{\partial
u}{\partial t}(t,\xi)=\mathcal{A}_1 u(t,\xi)+\bar{B}(u)(t,\xi)+
\bar{G}(u)(t,\xi)\,\frac{\partial
w^{Q_1}}{\partial t}(t,\xi), }\\
\vs \ds{u(0,\xi)=x(\xi),\ \ \ \xi \in\,D,\ \ \ \ \ \mathcal{N}_{1}
u\,(t,\xi)=0,\ \ \ t\geq 0,\ \ \ \xi \in\,\partial D,}
\end{array}\r.\end{equation} More precisely, we prove that for any $T>0$ and $\eta>0$
\begin{equation}
\label{limite} {\cal L}(u_\e) \rightharpoonup {\cal L}(\bar{u})\ \ \text{in}\ \ C([0,T];H),\ \ \ \ \ \ \text{as}\ \ \e\to 0,
\end{equation}
(see Theorem \ref{theo5.2}). Moreover, in the case the diffusion coefficient $g_1$ in the slow equation does not depend on the fast oscillating variable $v_\e$, we show that the convergence of $u_\e$ to $\bar{u}$ is in probability, that is for any $\eta>0$
\begin{equation}
\label{proby}
\lim_{\e\to 0}\Pro\le(|u_\e-\bar{u}|_{C([0,T];H)}>\eta\r)=0,
\end{equation}
(see Theorem \ref{averaging}).

In order to prove \eqref{limite}, we have to proceed in several
steps. First of all we show that the family
$\{\mathcal{L}(u_\e)\}_{\e \in\,(0,1]}$ is tight in
$\mathcal{P}(C([0,T];H))$ and this is obtained by a-priori bounds
for processes $u_\e$ in a suitable  H\"older norm with respect to time and in a suitable Sobolev norm 
with respect to space. We would like to stress that, as
we are only assuming \eqref{ops2bis} and \eqref{ops2} and not a law
of large numbers, we also need to prove a-priori bounds for the
conditioned momenta of $u_\e$.

Once we have the tightness of the family
$\{\mathcal{L}(u_\e)\}_{\e \in\,(0,1]}$, we have the weak
convergence of the sequence $\{\mathcal{L}(u_{\e_n})\}_{n
\in\,\nat}$, for some $\e_n\downarrow 0$,  to some probability
measure $\mathbb{Q}$ on $C_x([0,T]:H)$. The next steps consist in
identifying $\mathbb{Q}$ with $\mathcal{L}(\bar{u})$ and proving
that  limit \eqref{limite} holds. To this purpose we introduce the
martingale problem with parameters
$(x,\mathcal{A}_1,\bar{B},\bar{G},Q_1)$ and we show that
$\mathbb{Q}$ is a solution to such martingale problem. As the
coefficients $\bar{B}$ and $\bar{Q}$ are Lipschitz-continuous, we
have uniqueness and hence we can conclude that
$\mathbb{Q}=\mathcal{L}(\bar{u})$. This in particular implies that
for any $\e_n\downarrow 0$ the sequence
$\{\mathcal{L}(u_{\e_n})\}_{n \in\,\nat}$ converges weakly to
$\mathcal{L}(\bar{u})$ and hence \eqref{limite} holds. Moreover, in the case $g_1$ does not depend on $v_\e$, by a uniqueness argument this implies convergence in probability.

In the general case, the key point in the identification of $\mathbb{Q}$ with the
solution of the martingale problem associated with the averaged
equation \eqref{avintro} is the following limit
\[\lim_{\e\to
0}\E\le|\int_{t_1}^{t_2}\E \,\le(\mathcal{L}_{sl}\,\varphi(u_{\e}(r),v_{\e}(r))-
\mathcal{L}_{\text{av}}\,\varphi(u_{\e}(r))\le|\mathcal{F}_{t_1}\r.\r)\,dr\r|=0,\]
 where $\mathcal{L}_{sl}$ and $\mathcal{L}_{av}$ are
 the Kolmogorov operators associated respectively with the slow
motion equation, with frozen fast component, and with the averaged
equation, and $\{\mathcal{F}_t\}_{t\geq 0}$ is the filtration
associated with the noise. Notice that it is sufficient to check the validity of such a limit for any
cylindrical function $\varphi$ and any $0\leq t_1\leq t_2\leq T$.
The proof of the limit above is based on the Khasminskii argument
introduced in \cite{khas}, but it is clearly more delicate than in
\cite{khas}, as it concerns a system with an infinite number of
degrees of freedom (with all well known problems arising from
that).

In the particular case of $g_1$ not depending on $v_\e$, in order to prove \ref{proby} we do not need to pass through the martingale formulation. For any $h
\in\,D(A_1)$ we write
\[\begin{array}{l}
\ds{\le<u_\e(t),h\r>_H=\le<x,h\r>_H+\int_0^t\le<u_{\e}(s),A_1
h\r>_H\,ds+\int_0^t
\le<\bar{B}_1(u_{\e}(s)),h\r>_H\,ds}\\
\vs \ds{+\int_0^t \le<G_1(u_\e(s)) h,dw_1^{Q_1}(s)\r>_H+R_\e(t),}
\end{array}\]
where
\[R_\e(t):=\int_0^t
\le<B_1(u_{\e}(s),v_\e(s))-\bar{B}_1(u_{\e}(s)),h\r>_H\,ds,\] and,
by adapting the arguments introduced by Khasminskii in \cite{khas}
to the present infinite dimensional setting, we show that for any
$T>0$
\begin{equation}
\label{fine33} \lim_{\e\to 0}\,\E\sup_{t \in\,[0,T]}\,|R_\e(t)|=0.
\end{equation}
Thanks to the Skorokhod theorem and to a general argument due to
Gy\"ongy and Krylov (see \cite{gk}), this allows to obtain
\eqref{averaging}.

\section{Assumptions and preliminaries}
\label{sec2}

Let $D$ be a smooth  bounded domain of $\reals^d$, with $d\geq 1$.
Throughout the paper, we shall denote by $H$  the Hilbert space
$L^2(D)$, endowed with the usual scalar product
$\le<\cdot,\cdot\r>_H$ and with the corresponding norm
$|\cdot|_H$. The norm in $L^\infty(D)$ will be denoted by
$|\cdot|_0$.

We shall denote by $B_b(H)$ the Banach space of bounded Borel
functions $\varphi:H\to \reals$, endowed with the sup-norm
\[\|\varphi\|_0:=\sup_{x \in\,H}|\varphi(x)|.\]
$C_b(H)$ is  the subspace of uniformly continuous mappings
 and
$C^k_b(H)$ is the subspace of all $k$-times differentiable
mappings,  having bounded and uniformly continuous derivatives, up
to the $k$-th order, for $k \in\,\nat$. $C^k_b(H)$ is a Banach
space endowed with the norm
\[|\varphi|_k:=|\varphi|_0+\sum_{i=1}^k\sup_{x
\in\,H}|D^i\varphi(x)|_{\mathcal{L}^i(H)}=:|\varphi|_0+\sum_{i=1}^k\,[\varphi]_i,\]
where $\mathcal{L}^1(H):=H$ and, by recurrence,
$\mathcal{L}^i(H):=\mathcal{L}(H,\mathcal{L}^{i-1}(H))$, for any
$i>1$. Finally, we denote by $\text{Lip}(H)$ the set of functions
$\varphi:H\to$ such that
\[[\varphi]_{\text{Lip}(H)}:=\sup_{\substack{x,y \in\,H\\x\neq
y}}\frac{|\varphi(x)-\varphi(y)|}{|x-y|_H}<\infty.\]

We shall denote by $\mathcal{L}(H)$ the space of bounded
linear operators in $H$ and we shall denote by $\mathcal{L}_2(H)$ the subspace of
Hilbert-Schmidt operators, endowed with the norm
\[\|Q\|_{2}=\sqrt{\text{Tr}\,[Q^\star Q]}.\]

\medskip

The stochastic perturbations in the slow and in the fast
 motion
equations \eqref{eq0} are given respectively by the Gaussian
noises $\partial w^{Q_1}/\partial t(t,\xi)$ and $\partial
w^{Q_2}/\partial t(t,\xi)$, for $t\geq 0$ and $\xi \in\,D$, which
are assumed to be white in time and colored in space, in the case
of space dimension $d>1$. Formally, the cylindrical Wiener
processes $w^{Q_i}(t,\xi)$ are defined as the infinite sums
\[w^{Q_i}(t,\xi)=\sum_{k=1}^\infty Q_i e_{k}(\xi)\,\beta_{k}(t),\ \ \ \ i=1,2,\] where
 $\{e_{k}\}_{k \in\,\nat}$ is a
complete orthonormal basis in $H$, $\{\beta_{k}(t)\}_{k
\in\,\nat}$ is a sequence of mutually independent standard
Brownian motions defined on the same complete stochastic basis
$(\Omega,\mathcal{F}, \mathcal{F}_t, \mathbb{P})$ and $Q_i$ is a compact 
linear operator on $H$.

The operators $\mathcal{A}_1$ and $\mathcal{A}_2$ appearing
respectively in the slow and in the fast motion equation, are
second order uniformly elliptic operators, having continuous
coefficients on $D$, and the boundary operators $\mathcal{N}_1$
and $\mathcal{N}_2$ can be either the identity operator (Dirichlet
boundary condition) or a first order operator of the following
type
\[\sum_{j=1}^d \beta_j(\xi) D_j+\gamma(\xi)I,\ \ \ \ \ \xi \in\,\partial D,\]
for some $\beta_j, \gamma \in\,C^1(\bar{D})$ such that
\[\inf_{\xi \in\,\partial
D}\,\le|\le<\beta(\xi),\nu(\xi)\r>\r|>0,\] where $\nu(\xi)$ is the
unit normal at $\xi \in\,\partial D$ (uniform nontangentiality
condition).

The realizations $A_1$ and $A_2$ in $H$ of the differential
operators $\mathcal{A}_1$ and $\mathcal{A}_2$ , endowed respectively with the
boundary conditions $\mathcal{N}_1$ and $\mathcal{N}_2$, generate
two analytic semigroups $e^{t A_1}$ and $e^{t A_2}$, $t\geq 0$. In
what follows we shall assume that $A_1$, $A_2$ and $Q_1$, $Q_2$
satisfy the following conditions.

\begin{Hypothesis}
\label{H1} For $i=1, 2$ there exist a complete orthonormal system
$\{e_{i,k}\}_{k \in\,\nat}$ in $H$ and two sequences of
non-negative real numbers $\{\a_{i,k}\}_{k \in\,\nat}$ and
$\{\la_{i,k}\}_{k \in\,\nat}$,  such that
\[A_i\, e_{i,k}=-\a_{i,k}\, e_{i,k},\ \ \ \ \ Q_i e_{i,k}=\la_{i,k} e_{i,k},\ \ \ k \geq 1,\]
and
\begin{equation}
\label{esistenza} \kappa_i:=\sum_{k=1}^\infty
\la_{i,k}^{\rho_i}\,|e_{i,k}|_0^2<\infty,\ \ \ \ \
\zeta_i:=\sum_{k=1}^\infty
\a_{i,k}^{-\beta_i}\,|e_{i,k}|_0^2<\infty,
\end{equation}
for some  constants $\beta_i \in\,(0,+\infty)$ and $\rho_i
\in\,(2,+\infty]$ such that
\begin{equation}
\label{sola1} \,\frac{\beta_i(\rho_i-2)}{\rho_i}<1.
\end{equation}
Moreover,
\begin{equation}
\label{sola2} \inf_{k \in\,\nat}\,\a_{2,k}=:\la>0.
\end{equation}

\end{Hypothesis}

\begin{Remark}
{\em \begin{enumerate}
\item In several cases, as for example in the case of space dimension $d=1$ and in  the case of the
Laplace operator on a hypercube, endowed with Dirichlet boundary
conditions, the eigenfunctions $e_k$ are equi-bounded in the
sup-norm and then conditions \eqref{esistenza} become
\[\kappa_i=\sum_{k=1}^\infty
\la_{i,k}^{\rho_i}<\infty,\ \ \ \ \ \ \zeta_i=\sum_{k=1}^\infty
\a_{i,k}^{-\beta_i}<\infty,\] for positive constants $\beta_i,
\rho_i$ fulfilling \eqref{sola1}.  In general
\[|e_{i,k}|_\infty\leq c\,k^{a_i},\ \ \ \ \ \ k \in\,\nat,\]
for some $a_i\geq 0$. Thus,   conditions \eqref{esistenza} become
\[\kappa_i:=\sum_{k=1}^\infty
\la_{i,k}^{\rho_i}\,k^{2 a_i}<\infty,\ \ \ \ \ \
\zeta_i:=\sum_{k=1}^\infty \a_{i,k}^{-\beta_i}\,k^{2
a_i}<\infty.\]

\item For any reasonable domain $D\subset \rd$ one has
\[\a_{i,k}\sim k^{2/d},\ \ \ \ \ \ k \in\,\nat.\]
Thus, if the eigenfunctions $e_k$ are equi-bounded in the
sup-norm, we have \[\zeta_i\leq c\sum_{k=1}^\infty
\a_{i,k}^{-\beta_i}\sim\sum_{k=1}^\infty k^{-\frac{2 \beta_i}d}.\]
This means  that, in order to have $\zeta_i<\infty$, we need
\[\beta_i>\frac d 2.\]
In particular, in order to have also $\kappa_i<\infty$ and
condition \eqref{sola1} satisfied, in space dimension $d=1$ we can
take $\rho_i=+\infty$, so that we can deal with white noise, both
in time and in space. In space dimension $d=2$ we can take any
$\rho_i<\infty$ and in space dimension $d\geq 3$ we need
\[\rho_i<\frac{2d}{d-2}.\]
In any case, notice that  it is never required to take $\rho_i=2$,
which means to have a noise with trace-class covariance. To
this purpose, it can be useful to compare these conditions with
Hypothesis 2 and Hypothesis 3 in \cite{cerrai1}.

\end{enumerate}}
\end{Remark}

\medskip

As far as the coefficients $b_1, b_2$ and $g_1, g_2$ are
concerned, we assume the following conditions.

\begin{Hypothesis}
\label{H2}
\begin{enumerate}
\item The mappings $b_i:D\times \mathbb{R}^2\to
\mathbb{R}$ and $g_i:D\times \mathbb{R}^2\to \mathbb{R}$ are
measurable, both for $i=1$ and $i=2$, and  for almost all $\xi
\in\,D$ the mappings $b_i(\xi,\cdot):\mathbb{R}^2\to \mathbb{R}$
and $g_i(\xi,\cdot):\mathbb{R}^2\to \mathbb{R}$ are
Lipschitz-continuous, uniformly with respect to $\xi \in\,D$. Moreover
\[\sup_{\xi \in\,D}|b_2(\xi,0,0)|<\infty.\]

\item It holds
\begin{equation}
\label{uffa1}
\sup_{\substack{\xi\in\,D\\\si_1 \in\,\reals}}\,\sup_{\substack{\si_2,\rho_2 \in\,\mathbb{R}\\
\si_2 \neq
\rho_2}}\,\frac{|b_2(\xi,\si_1,\si_2)-b_2(\xi,\si_1,\rho_2)|}{|\si_2-\rho_2|}=:L_{b_2}<\la,
\end{equation}
where $\la$ is the constant introduced in \eqref{sola2}.

\item There exists  $\gamma<1$  such that
\begin{equation}
\label{growthg2} \sup_{\xi \in\,D}\,|g_2(\xi,\si)|\leq
c\,\le(1+|\si_1|+|\si_2|^\gamma\r),\ \ \ \ \ \si=(\si_1,\si_2)
\in\,\reals^2.
\end{equation}
\end{enumerate}

\end{Hypothesis}

\begin{Remark}
{\em \begin{enumerate}
\item Notice that condition \eqref{growthg2} on the
growth of $g_2(\xi,\si_1,\cdot)$ could be replaced with the condition
\[\sup_{\substack{\xi \in\,D\\ \si_1
\in\,\reals}}[g_2(\xi,\si_1,\cdot)]_{\text{Lip}}\leq \eta,\] for
some $\eta$ sufficiently small.

\item In what follows we shall set 
\[ L_{g_2}:=\sup_{\substack{\xi\in\,D\\\si_1 \in\,\reals}}\,\sup_{\substack{\si_2,\rho_2 \in\,\mathbb{R}\\
\si_2 \neq
\rho_2}}\,\frac{|g_2(\xi,\si_1,\si_2)-g_2(\xi,\si_1,\rho_2)|}{|\si_2-\rho_2|}.\]
\end{enumerate}
}
\end{Remark}

In what follows we shall set
\[B_i(x,y)(\xi):=b_i(\xi,x(\xi),y(\xi))\]
and
\[[G_i(x,y)z](\xi):=g_i(\xi,x(\xi),y(\xi))z(\xi),\]
for any $\xi\in\,D$, $x, y, z \in\,H$ and $i=1,2$. Due to
Hypothesis \ref{H2}, the mappings
\[(x,y) \in\,H\times H\mapsto
B_i(x,y) \in\, H,\] are Lipschitz-continuous, as well as the
mappings
\[(x,y) \in\,H\times H\mapsto G_i(x,y)
\in\,\mathcal{L}(H; L^1(\D))\] and
\[(x,y) \in\,H\times H\mapsto G_i(x,y)
\in\,\mathcal{L}(L^\infty(D); H).\]

\medskip

Now, for any fixed $T>0$ and $p\geq 1$, we  denote by $\mathcal{H}_{T,p}$ the
space of processes in $C([0,T];L^p(\Omega;H))$, which are
adapted to the filtration $\{\F_t\}_{t\geq 0}$ associated with  the noise.
$\mathcal{H}_{T,p}$ is a Banach space, endowed with the norm
\[\|u\|_{\mathcal{H}_{T,p}}=\le(\,\sup_{t
\in\,[0,T]}\,\E\,|u(t)|_H^p\r)^{\frac 1p}.\] Moreover, we denote
by $\mathcal{C}_{T,p}$ the subspace of processes $u
\in\,L^p(\Omega;C([0,T];H))$, endowed with the norm
\[\|u\|_{\mathcal{C}_{T,p}}=\le(\,\E\,\sup_{t
\in\,[0,T]}\,|u(t)|_H^p\r)^{\frac 1p}.\]

\medskip

With all notations we have introduced, system
\eqref{eq0} can be rewritten as the following abstract system
\begin{equation}
\label{astratto} \le\{\begin{array}{l} \ds{du_\e(t)=\le[A_1
u_\e(t)+B_1(u_\e(t),v_\e(t))\r]\,ds+G_1(u_\e(t),v_\e(t))\,dw^{Q_1}(t),\ \
\ \ u_\e(0)=x}\\
\vs \ds{dv_\e(t)=\frac 1\e \le[A_2
v_\e(t)+B_2(u_\e(t),v_\e(t))\r]\,ds+\frac 1{\sqrt{\e}}\,
G_2(u_\e(t),v_\e(t))\,dw^{Q_2}(t),\ \ \ \ v_\e(0)=y.}
\end{array}\r.
\end{equation}

As known from the existing literature (see for example
\cite{dpz1}), according to Hypotheses \ref{H1} and
\ref{H2} for any $\e>0$ and $x,y \in\,H$ and for any
$p \geq 1$ and $T>0$ there exists a unique mild solution
$(u_\e,v_\e) \in\,\mathcal{C}_{T,p}\times \mathcal{C}_{T,p}$ to
system \eqref{eq0}. This means that there exist  two processes
$u_\e$ and $v_\e$  in $\mathcal{C}_{T,p}$, which are unique, such
that
\[ u_\e(t)=e^{t A_1}x+\int_0^t e^{(t-s)A_1}
B_1(u_\e(s),v_\e(s))\,ds+\int_0^t e^{(t-s)A_1}
G_1(u_\e(s),v_\e(s))\,dw^{Q_1}(s),\] and
\[v_\e(t)=e^{t A_2/\e}y+\frac 1\e \int_0^t e^{(t-s)\frac{A_2}\e} B_2(u_\e(s),v_\e(s))\,ds+\frac 1{\sqrt{\e}}
\int_0^t e^{(t-s)\frac{A_2}\e}G_2(u_\e(s),v_\e(s))\,dw^{Q_2}(s).\]

\subsection{The fast motion equation}
\label{subsec21}

For any fixed $x \in\, H$, we consider the problem
\begin{equation}
\label{fast} \le\{
\begin{array}{l}
\ds{\frac{\partial v}{\partial t}(t,\xi)=\mathcal{A}_2
v(t,\xi)+b_2(\xi,x(\xi),v(t,\xi))+g_2(\xi,x(\xi),v(t,\xi))\,\frac{\partial
w^{Q_2}}{\partial t}(t,\xi),}\\
\vs \ds{v(0,\xi)=y(\xi),\ \ \ \xi \in\,D, \ \ \ \ \ \ \
\mathcal{N}_{2} v\,(t,\xi)=0,\ \ \ \ t\geq 0,\ \ \ \ \xi
\in\,\partial D.}
\end{array}\r.
\end{equation}
Under Hypotheses \ref{H1} and \ref{H2}, such a problem
 admits a unique mild solution $v^{x,y}
\in\,\mathcal{C}_{T,p}$, for any $T>0$ and $p\geq 1$, and for any
fixed frozen slow variable $x \in\,H$ and any initial condition $y
\in\,H$ (for a proof see e.g. \cite[Theorem 5.3.1]{dpz2}).

By arguing as in the proof of \cite[Theorem 7.3]{cerrai2}, is it
possible to show that there exists some $\d_1>0$ such that for any
$p\geq 1$
\begin{equation}
\label{sola5} \E\,|v^{x,y}(t)|^p_H\leq c_p\,(1+|x|_H^p+e^{-\d_1 p
t}\,|y|_H^p),\ \ \ \ \ t\geq 0.
\end{equation}
 In particular, as shown in \cite{cerrai2}, this
implies that there exists some $\theta>0$ such that for any $a>0$
\begin{equation}
\label{tight} \sup_{t\geq
a}\,\E\,|v^{x,y}(t)|_{D((-A_2)^\theta)}\leq c_a(1+|x|_H+|y|_H).
\end{equation}

Now, for any $x \in\,H$ we denote by $P^x_t$ the transition semigroup
associated with problem \eqref{fast}, which is defined by
\[P^x_t \varphi(y)=\E\,\varphi(v^{x,y}(t)),\ \ \ \ t\geq 0,\ \ y
\in\,H,\] for any $\varphi \in\,B_b(H)$. Due to \eqref{tight}, the
family $\{\mathcal{L}(v^{x,y}(t))\}_{t\geq a}$ is tight in
$\mathcal{P}(H,\mathcal{B}(H))$ and then by the Krylov-Bogoliubov
theorem,  there exists an invariant measure $\mu^x$ for the
semigroup $P^x_t$. Moreover, due to \eqref{sola5} for any $p\geq
1$ we have
\begin{equation}
\label{sola6} \int_H |z|^p_H\,\mu^x(dz)\leq c_p\,(1+|x|_H^p)
\end{equation}
(for a proof see \cite[Lemma 3.4]{cf}).

 As in \cite[Theorem
7.4]{cerrai2}, it is possible to show that if $\la$ is
sufficiently large and/or $L_{b_2}$, $L_{g_2}$, $\zeta_2$ and
$\kappa_2$ are sufficiently small, then there exist some $c,
\d_2>0$ such that
\begin{equation}
\label{sola3} \sup_{x \in\,H}\,\E\,|v^{x,y_1}(t)-v^{x,y_2}(t)|_H\leq c\,e^{-\d_2
t}\,|y_1-y_2|_H,\ \ \ \ \ t\geq 0,
\end{equation}
for any $y_1,y_2 \in\,H$. In particular, this implies that
$\mu^x$ is the unique invariant measure for $P^x_t$ and  is
strongly mixing. Moreover, by arguing as in \cite[Theorem 3.5 and
Remark 3.6]{cf}, from \eqref{sola6} and \eqref{sola3} we have
\begin{equation}
\label{sola4} \le|P^x_t \varphi(y)-\int_H
\varphi(z)\,\mu^x(dz)\r|\leq c\,\le(1+|x|_H+|y|_H\r)\,e^{-\d_2
t}\,[\varphi]_{\text{Lip}(H)},
\end{equation}
for any $x,y \in\,H$ and $\varphi \in\,\text{Lip}(H)$. In
particular, this implies the following fact.
\begin{Lemma}
\label{conv}
 Under the above conditions, for any $\varphi
\in\,\text{{\em Lip}}(H)$, $T>0$, $x,y \in\,H$  and $t\geq 0$
\begin{equation}
\label{sola9} \E\,\le|\frac
1T\int_t^{t+T}\varphi(v^{x,y}(s))\,ds-\int_H
\varphi(z)\,\mu^x(dz)\r|\leq
\frac{c}{\sqrt{T}}\,\le(H_\varphi(x,y)+|\varphi(0)|\r),\end{equation}
 where
\begin{equation}
\label{sola10} H_\varphi(x,y):=[\varphi]_{\text{{\em
Lip}}(H)}(1+|x|_H+|y|_H).
\end{equation}
\end{Lemma}

\begin{proof}
We have
\[\begin{array}{l}
\ds{\E\,\le(\,\frac
1T\int_t^{t+T}\varphi(v^{x,y}(s))\,ds-\bar{\varphi}^x\r)^2}\\
\vs \ds{=\frac
1{T^2}\,\int_t^{t+T}\int_t^{t+T}\E\,\le(\varphi(v^{x,y}(s))-\bar{\varphi}^x\r)
\le(\varphi(v^{x,y}(r))-\bar{\varphi}^x\r)\,ds\,dr}\\
\vs \ds{=\frac
2{T^2}\int_t^{t+T}\int_r^{t+T}\E\,\le(\varphi(v^{x,y}(s))-\bar{\varphi}^x\r)
\le(\varphi(v^{x,y}(r))-\bar{\varphi}^x\r)\,ds\,dr,}
\end{array}\]
where
\[\bar{\varphi}^x:=\int_H
\varphi(z)\,\mu^x(dz).\]
From the Markovianity of $v^{x,y}(t)$,
for $r\leq s$ we have
\[\E\,\le(\varphi(v^{x,y}(s))-\bar{\varphi}^x\r)
\le(\varphi(v^{x,y}(r))-\bar{\varphi}^x\r)=\E\,\le[\le(\varphi(v^{x,y}(r))-\bar{\varphi}^x\r)
\,P_{s-r}\le(\varphi(v^{x,y}(r))-\bar{\varphi}^x\r)\r],\] so that,
in view of  \eqref{sola5} and \eqref{sola4},
\[\begin{array}{l}
\ds{\E\,\le(\,\frac
1T\int_t^{t+T}\varphi(v^{x,y}(s))\,ds-\bar{\varphi}^x\r)^2}\\
\vs \ds{\leq \frac
c{T^2}\int_t^{t+T}\int_r^{t+T}\le([\varphi]_{\text{
Lip}(H)}\le(\E\,|v^{x,y}(r)|_H^2\r)^{\frac
12}+|\varphi(0)|+|\bar{\varphi}^x|\r)}\\
\vs \ds{ \times
\le(\E\,\le[P_{s-r}\varphi(v^{x,y}(r))-\bar{\varphi}^x\r]^2\r)^{\frac
12}\,ds\,dr} \\
\vs \ds{\leq \frac
c{T^2}\,\le(H_\varphi(x,y)+|\varphi(0)|+|\bar{\varphi}^x|\r)\,H_\varphi(x,y)\int_t^{t+T}\int_r^{t+T}e^{-\d_2(s-r)}\,ds\,dr}\\
\vs \ds{\leq \frac
c{T}\,\le(H_\varphi(x,y)+|\varphi(0)|+|\bar{\varphi}^x|\r)\,H_\varphi(x,y),}
\end{array}\]
with $H_\varphi(x,y)$ defined as in \eqref{sola10}. As from
\eqref{sola6} we have
\[|\bar{\varphi}^x|\leq [\varphi]_{\text{Lip}(H)}
(1+|x|_H)+|\varphi(0)|,\] we can conclude that \eqref{sola9}
holds.

\end{proof}

\subsection{The averaged coefficients}
\label{subsec22}

In the next hypotheses  we introduce the coefficients
of the {\em averaged} equation, and we give conditions which
assure the convergence of the slow motion component $u_\e$ to its
solution. For the reaction coefficient we assume the following
condition.

\begin{Hypothesis}
\label{H4}  There exists a Lipschitz-continuous mapping
$\bar{B}:H\to H$ such that for any $T>0$, $t\geq 0$ and $x,y,h
\in\,H$
\begin{equation}
\label{sola7} \le|\frac
1T\int_t^{t+T}\E\,\le<B_1(x,v^{x,y}(s)),h\r>_H\,ds-\le<\bar{B}(x),h\r>_H\r|\leq
\a(T)\,\le(1+|x|_H+|y|_H\r)\,|h|_H,\end{equation}
 for some function $\a(T)$ such that
\[\lim_{T\to \infty}\a(T)=0.\]

\end{Hypothesis}

Concerning the diffusion coefficient, we assume the following condition.

\begin{Hypothesis}
\label{H5}  There exists a Lipschitz-continuous mapping
$\bar{G}:H\to \mathcal{L}(L^\infty(D);H)$ such that for any $T>0$,
$t\geq 0$, $x,y \in\,H$ and $h,k \in\,L^\infty(D)$
\begin{equation}
\label{sola8} \begin{array}{l} \ds{\le|\frac
1T\int_t^{t+T}\E\,\le<G_1(x,v^{x,y}(s))h,
G_1(x,v^{x,y}(s))k\r>_H\,ds-\le<\bar{G}(x)h,\bar{G}(x)k\r>_H\r|}\\
\vs \ds{\leq \a(T)\,\le(1+|x|^2_H+|y|^2_H\r)\,|h|_\infty
|k|_\infty,}
\end{array}\end{equation}
 for some $\a(T)$ such that
\[\lim_{T\to \infty}\a(T)=0.\]
\end{Hypothesis}

\section{The averaged equation}
\label{sec3}

In this section we describe some relevant situations in
which Hypotheses \ref{H4} and \ref{H5} are verified and we give some notations and some results about the martingale problem and the mild solution for the averaged equation.

\subsection{The reaction coefficient $\bar{B}$}
\label{subsec3.1}

For any fixed $x, h \in\,H$, the mapping
\[y \in\,H\mapsto \le<B_1(x,y),h\r>_H \in\,\mathbb{R},\]
is Lipschitz-continuous. Then, if we define
\[\bar{B}(x):=\int_H B_1(x,z)\,\mu^x(dz),\ \ \ \ x \in\,H,\]
thanks to \eqref{sola4} we have that limit \eqref{sola7} holds, with $\a(T)=c/\sqrt{T}$.

Due to \eqref{sola7}, for any $x_1,x_2,y,h \in\,H$ we have
\[\begin{array}{l}
\ds{\le<\bar{B}_1(x_1)-\bar{B}_1(x_2),h\r>_H
=\lim_{T\to \infty}\frac
1T\,\E\int_0^T\le<B_1(x_1,v^{x_1,y}(s))-B_1(x_2,v^{x_2,y}(s)),h\r>_H\,ds.}
\end{array}\]
Then, as the mapping $B_1:H\times H\to H$ is Lipschitz continuous,
we have
\[\begin{array}{l}
\ds{\le|\le<\bar{B}_1(x_1)-\bar{B}_1(x_2),h\r>_H\r|\leq
c\,|h|_H\limsup_{T\to \infty}\frac
1T\int_0^T\le(|x_1-x_2|_H+\E\,|v^{x_1,y}(s)-v^{x_2,y}(s)|_H\r)\,ds}\\
\vs \ds{=c\,|h|_H\le(|x_1-x_2|_H+\limsup_{T\to \infty}\frac
1T\int_0^T \E\,|\rho(s)|_H\,ds\r),}
\end{array}\]
where $\rho(t):=v^{x_1,y}(t)-v^{x_2,y}(t)$, for any $t\geq 0$. In
the next lemma we show that under suitable conditions on the coefficients there exists some constant $c>0$ such
that for any $T>0$
\[\frac
1T\int_0^T \E\,|\rho(s)|_H\,ds\leq c\,|x_1-x_2|_H.\] Clearly, this
implies the Lipschitz continuity of $\bar{B}_1$.

\begin{Lemma}
\label{rho} 
Assume that 
\begin{equation}
\label{condition} \le(\frac
{L_{b_2}}{\la}\r)^2+L_{g_2}^2\,\le(\frac {\beta_2}
e\r)^{\frac{\beta_2(\rho_2-2)}{\rho_2}}
\zeta_2^{\frac{\rho_2-2}{\rho_2}} \kappa_2^\frac 2{\rho_2}
\le(\frac{\rho_2}{\la(\rho_2+2)}\r)^{1-\frac{\beta_2(\rho_2-2)}{\rho_2}}=:M_0<\frac
12.
\end{equation}

Then, under Hypotheses \ref{H1} and \ref{H2},
there exists $c>0$ such that for any $x_1,x_2, y\in\,H$ and $t>0$
\[\frac
1t\int_0^t \E\,|v^{x_1,y}(s)-v^{x_2,y}(s)|_H\,ds\leq
c\,|x_1-x_2|_H.\]
\end{Lemma}

\begin{proof}
We set $\rho(t):=v^{x_1,y}(t)-v^{x_2,y}(t)$ and we define
\[\Gamma(t):=\int_0^t
e^{(t-s)A_2}\le[G_2(x_1,v^{x_1,y}(s))-G_2(x_2,v^{x_2,y}(s))\r]\,dw^{Q_2}(s),\]
and set $\La(t):=\rho(t)-\Gamma(t)$. For any $\eta \in\,(0,\la/2)$
we can fix $c_{1,\eta}>0$ such that
\[\begin{array}{l}
\ds{\frac 12 \frac d{dt}\,|\La(t)|_H^2\leq
-\la\,|\La(t)|_H^2+\le(c\,|x_1-x_2|_H+L_{b_2}\,|\rho(t)|_H\r)\,|\La(t)|_H}\\
\vs \ds{\leq -\le(\frac{\la}2-\eta\r)\,|\La(t)|_H^2+\frac
{L_{b_2}^2}{2\la}\,|\rho(t)|_H^2+c_{1,\eta}\,|x_1-x_2|^2.}
\end{array}\] This implies
\[|\La(t)|_H^2\leq
\le(1+\frac{2 c_{1,\eta}}{\la-2\eta}\r)\,|x_1-x_2|^2+\frac
{L_{b_2}^2}{\la}\int_0^t e^{-(\la-2\eta)(t-s)}|\rho(s)|_H^2\,ds,\]
so that
\begin{equation}
\label{sola15} \begin{array}{l} \ds{\E\,|\rho(t)|_H^2\leq
2\le(1+\frac{2 c_{1,\eta}}{\la-2\eta}\r)\,|x_1-x_2|^2
+ \frac{2L_{b_2}^2}{\la}\int_0^t
e^{-(\la-2\eta)(t-s)}\E\,|\rho(s)|_H^2\,ds+2\,\E\,|\Gamma(t)|_H^2.}
\end{array}\end{equation}

Thanks to Hypothesis \ref{H1}, for any $J
\in\,\mathcal{L}(L^\infty(D),H)\cap \mathcal{L}(H,L^1(D))$, with
$J=J^\star$ and for any $s\geq 0$ we have
\[\begin{array}{l}
\ds{\|e^{sA_2} J Q_2\|_{2}^2=\sum_{k=1}^\infty \la_{2,k}^2
\le|e^{sA_2}
J e_{2,k}\r|_H^2}\\
\vs \ds{\leq \le(\,\sum_{k=1}^\infty
\la_{2,k}^{\rho_2}\,|e_{2,k}|_0^2\r)^{\frac
2{\rho_2}}\le(\,\sum_{k=1}^\infty \le|e^{sA_2} J
e_{2,k}\r|_H^{\frac{2\rho_2}{\rho_2-2}}\,|e_{2,k}|_H^{-\frac
4{\rho_2-2}}\r)^{\frac{\rho_2-2}{\rho_2}}}\\
\vs \ds{\leq \kappa_2^{\frac 2{\rho_2}}\sup_{k
\in\,\nat}\,\le|e^{sA_2} J e_{2,k}\r|_H^{\frac
4{\rho_2}}|e_{2,k}|_0^{-\frac 4{\rho_2}}\le(\,\sum_{k=1}^\infty
\le|e^{sA_2} J e_{2,k}\r|_H^2\r)^{\frac{\rho_2-2}{\rho_2}}.}
\end{array}
\]
Hence, thanks to \eqref{sola2}, we obtain
\begin{equation}
\label{sola18}
\begin{array}{l} \ds{\|e^{sA_2} J Q_2\|_{2}^2\leq
\kappa_2^{\frac
2{\rho_2}}\,\|J\|_{\mathcal{L}(L^\infty(D),H)}^{\frac
4{\rho_2}}e^{-\frac{4\la}{\rho_2} s}\le(\,\sum_{k=1}^\infty
\le|e^{sA_2} J e_{2,k}\r|_H^2\r)^{\frac{\rho_2-2}{\rho_2}}.}
\end{array}
\end{equation}
We have
\[\begin{array}{l}
\ds{\sum_{k=1}^\infty \le|e^{sA_2} J
e_{2,k}\r|_H^2=\sum_{k=1}^\infty\,\sum_{h=1}^\infty\le|\le<e^{sA_2}
J e_{2,k},e_{2,h}\r>_H\r|^2=\sum_{h=1}^\infty\,\sum_{k=1}^\infty
\le|\le< e_{2,k},J e^{sA_2}
e_{2,h}\r>_H\r|^2}\\
\vs \ds{=\sum_{h=1}^\infty\,|J e_{2,h}|_H^2 e^{-2\a_{2,h} s}\leq
e^{-\la s}\|J\|_{\mathcal{L}(L^\infty(D),H)}^2\sum_{h=1}^\infty
|e_{2,h}|_0^2 e^{-\a_{2,h} s}.}
\end{array}\]
Then, as for any $\beta>0$
\begin{equation}
\label{beta} e^{-\a t}\leq \le(\frac \beta
e\r)^{\beta}\,t^{-\beta} \a^{-\beta},\ \ \ \ \ \a, t>0,
\end{equation}
if we take $\beta_2$ as in condition \eqref{esistenza}, we get
\[\begin{array}{l}
\ds{\sum_{k=1}^\infty \le|e^{sA_2} J e_{2,k}\r|_H^2\leq \le(\frac
{\beta_2} e\r)^{\beta_2}s^{-\beta_2}\,e^{-\la
s}\|J\|_{\mathcal{L}(L^\infty(D),H)}^2\sum_{h=1}^\infty
|e_{2,h}|_0^2\a_h^{-\beta_2}}\\
\vs \ds{\leq \le(\frac {\beta_2} e\r)^{\beta_2}\zeta_2\,
s^{-\beta_2}\,e^{-\la s}\|J\|_{\mathcal{L}(L^\infty(D),H)}^2,}
\end{array}\] and  from \eqref{sola18} we can conclude
\begin{equation}
\label{sola20} \|e^{sA_2} J Q_2\|_{2}^2\leq \le(\frac {\beta_2}
e\r)^{\frac{\beta_2(\rho_2-2)}{\rho_2}}
\zeta_2^{\frac{\rho_2-2}{\rho_2}} \kappa_2^\frac
2{\rho_2}\,s^{-\frac{\beta_2(\rho_2-2)}{\rho_2}}
e^{-\frac{\la(\rho_2+2)}{\rho_2}
s}\|J\|_{\mathcal{L}(L^\infty(D),H)}^2. \end{equation}
 This means
that if we set \[K_2:=\le(\frac {\beta_2}
e\r)^{\frac{\beta_2(\rho_2-2)}{\rho_2}}
\zeta_2^{\frac{\rho_2-2}{\rho_2}} \kappa_2^\frac 2{\rho_2},\] and
if we take
\[J:=G_2(x_1,v^{x_1,y}(s))-G_2(x_2,v^{x_2,y}(s)),\]
for
any $\eta>0$ we can find some $c_{2,\eta}>0$ such that
\[\begin{array}{l}
\ds{\E\,|\Gamma(t)|_H^2=\int_0^t\E\,\|e^{(t-s)A_2}\le[G_2(x_1,v^{x_1,y}(s))-G_2(x_2,v^{x_2,y}(s))Q_2\r]\|_{2}^2\,ds}\\
\vs \ds{\leq K_2\int_0^t
(t-s)^{-\frac{\beta_2(\rho_2-2)}{\rho_2}}e^{-\frac{\la(\rho_2+2)}{\rho_2}
(t-s)}\E\,\le(c\,|x_1-x_2|_H+L_{g_2}\,|\rho(s)|_H\r)^2\,ds}\\
\vs \ds{\leq
c_{2,\eta}\,|x_1-x_2|^2_H+(1+\eta)L_{g_2}^2\,K_2\int_0^t
(t-s)^{-\frac{\beta_2(\rho_2-2)}{\rho_2}}e^{-\frac{\la(\rho_2+2)}{\rho_2}
(t-s)}\E\,|\rho(s)|_H^2\,ds,}
\end{array}\]
last inequality following from the fact that, according to
\eqref{sola1}, $\beta_2(\rho_2-2)/\rho_2<1$.

Now, if we plug the inequality above into \eqref{sola15}, for any
$\eta\leq \eta_1$ we obtain
\[\begin{array}{l}
\ds{\E\,|\rho(t)|_H^2\leq 2\le(1+\frac{2
\,c_{1,\eta}}{\la-2\eta}\r)\,|x_1-x_2|^2+
\frac{2L_{b_2}^2}{\la}\int_0^t
e^{-(\la-2\eta)(t-s)}\E\,|\rho(s)|_H^2\,ds}\\
\vs
\ds{+2\,c_{2,\eta}\,|x_1-x_2|^2_H+2\,(1+\eta)L_{g_2}^2\,K_2\int_0^t
(t-s)^{-\frac{\beta_2(\rho_2-2)}{\rho_2}}e^{-\frac{\la(\rho_2+2)}{\rho_2}
(t-s)}\E\,|\rho(s)|_H^2\,ds,}
\end{array}\]
and hence, if we integrate with respect to $t$ both sides, from
the Young inequality we get
\[\begin{array}{l}
\ds{\int_0^t \E\,|\rho(s)|_H^2\,ds\leq 2\le(1+\frac{2
\,c_{1,\eta}}{\la-2\eta}+c_{2,\eta}\r)t\,|x_1-x_2|_H^2}\\
\vs \ds{+2\le[ \frac{L_{b_2}^2}{\la}\int_0^t
e^{-(\la-2\eta)s}\,ds+(1+\eta)L_{g_2}^2\,K_2\int_0^t
s^{-\frac{\beta_2(\rho_2-2)}{\rho_2}}e^{-\frac{\la(\rho_2+2)}{\rho_2}
s}\,ds\r]\!\int_0^t\!\!\E\,|\rho(s)|_H^2 ds}\\
\vs \ds{\leq c_{3,\eta}
t\,|x_1-x_1|_H^2+2M_\eta\int_0^t\E\,|\rho(s)|_H^2 ds,}
\end{array}\] where
\begin{equation}
\label{sola21} M_\eta:=\frac
{L_{b_2}^2}{\la(\la-2\eta)}+(1+\eta)L_{g_2}^2\,\le(\frac {\beta_2}
e\r)^{\frac{\beta_2(\rho_2-2)}{\rho_2}}
\zeta_2^{\frac{\rho_2-2}{\rho_2}} \kappa_2^\frac 2{\rho_2}
\le(\frac{\rho_2}{\la(\rho_2+2)}\r)^{1-\frac{\beta_2(\rho_2-2)}{\rho_2}}.
\end{equation}
Now, as in \eqref{condition} we have assumed that $M_0<1/2$, we
can fix $\bar{\eta} \in\,(0,\la/2)$ such that $2\,M_{\bar{\eta}}<1$,
and hence
\[\int_0^t \E\,|\rho(s)|_H^2\,ds\leq
\frac{c_{3,\bar{\eta}}}{1-2 M_{\bar{\eta}}}\, t\,|x_1-x_1|_H^2.\]
This implies
\[\frac 1{t}\int_0^t \E\,|\rho(s)|_H\,ds\leq
\le(\frac{c_{3,\bar{\eta}}}{1-2M_{\bar{\eta}}}\r)^{\frac 12}\,
|x_1-x_1|_H,\] and  the proof of the lemma is finished.
\end{proof}

\subsection{The diffusion coefficient $\bar{G}$}
\label{subsec3.2}

If we assume that the function $g_1:D\times \mathbb{R}^2\to
\mathbb{R}$ is uniformly bounded, the mapping $G_1$ is well
defined from $H$ into $\mathcal{L}(H)$. Moreover, for any fixed
$x, h, k \in\,H$ the mapping \[z \in\,H\mapsto \le<G_1(x,z)
h,G_1(x,z) k\r>_H \in\,\reals,\] is Lipschitz continuous. Thus,
under the assumptions described above,
 if we take
\begin{equation}
\label{limdef} \le<S(x)h,k\r>_H=\int_H\le<G_1(x,z)h, G_1(x,z)
k\r>_H\,\mu^x(dz),
\end{equation}
we have that  $S:H\to \mathcal{L}(H)$ and, due to \eqref{sola4},  for any $T>0$, $t\geq 0$
and $x,y,h,k \in\,H$
\begin{equation} \label{defy}
\begin{array}{l}
\ds{\le|\frac 1T\int_t^{t+T}\E\,\le<G_1(x,v^{x,y}(s))h,
G_1(x,v^{x,y}(s))k\r>_H\,ds-\le<S(x)h,k\r>_H\r|}\\
\vs \ds{\leq \a(T)\,\le(1+|x|^2_H+|y|^2_H\r)\,|h|_H\,|k|_H,}
\end{array}
\end{equation}
 for some function  $\a(T)$ going to zero as $T\uparrow \infty$.

  It is immediate to check that  $S(x)=S(x)^\star$ and
$S(x)\geq 0$, for any $x \in\,H$. Then, as well known, there
exists an operator $\bar{G}(x) \in\,\mathcal{L}(H)$ such that
$\bar{G}(x)^2=S(x)$. If we assume that there exists $\d>0$ such
that
\[\inf_{\substack{\xi \in\,D\\\si
\in\,\mathbb{R}^2}}g_1(\xi,\si)\geq \d,\] we have that $S(x)\geq
\d^2$, and hence $\bar{G}(x)\geq \d$. In particular,
$\bar{G}(x)$ is invertible and
\[\|\bar{G}(x)^{-1}\|_{\mathcal{L}(H)}\leq \frac{1}{\d}.\]
Next, we notice that for any $x_1, x_2 \in\,H$
\begin{equation}
\label{comm} S(x_1) S(x_2)=S(x_2) S(x_1).\end{equation} Actually,
according to \eqref{limdef} for any $h,k \in\,H$
\[\begin{array}{l}
\ds{\le<S(x_1) S(x_2) h,k\r>_H=\int_H\le<G_1(x_1,z)S(x_2)h,
G_1(x_1,z) k\r>_H\,\mu^{x_1}(dz)}\\
\vs \ds{=\int_H \int_H \le<G_1^2(x_2,w)h, G_1^2(x_1,z)
k\r>_H\,\mu^{x_2}(dw)\,\mu^{x_1}(dz)}\\
\vs \ds{=\int_H \int_H \le<G_1^2(x_1,z)h, G_1^2(x_2,w)
k\r>_H\,\mu^{x_1}(dz)\,\mu^{x_2}(dw)=\le<S(x_2) S(x_1) h,k\r>_H.}
\end{array}\]
In particular, from \eqref{comm} for any $x_1, x_2 \in\,H$ we have
\begin{equation}
\label{commg} \bar{G}(x_1)\bar{G}(x_2)=\bar{G}(x_2) \bar{G}(x_1).
\end{equation}

Now, as $g_1(\xi,\cdot):\mathbb{R}^2\to \mathbb{R}^2$ is bounded
and  Lipschitz-continuous, uniformly with respect to $\xi \in\,D$,
we have that $g^2_1(\xi,\cdot):\mathbb{R}^2\to \mathbb{R}^2$ is
Lipschitz-continuous as well, uniformly with respect to $\xi
\in\,D$. This implies that for any $x_1, x_2, y \in\,H$, $h
\in\,L^\infty(D)$ and $k \in\,H$
\[\begin{array}{l}
\ds{\le|\le<\le[G^2_1(x_1,v^{x_1,y}(s))-G^2_1(x_2,v^{x_2,y}(s))\r]h,k\r>_H\r|}\\
\vs \ds{\leq
c\le(|x_1-x_2|_H+|v^{x_1,y}(s)-v^{x_2,y}(s)|_H\r)|h|_\infty
|k|_H,} \end{array}\] so that, according to \eqref{defy}
\[\begin{array}{l}
\ds{\le|\le<(S(x_1)-S(x_2))h,k\r>_H\r|}\\
\vs \ds{\leq c\,|h|_\infty |k|_H \limsup_{T\to \infty} \frac
1T\int_0^T
\le(|x_1-x_2|_H+\E\,|v^{x_1,y}(s)-v^{x_2,y}(s)|_H\r)\,ds}\\
\vs \ds{=c\,|h|_\infty |k|_H\le(|x_1-x_2|_H+ \limsup_{T\to \infty}
\frac 1T\int_0^T\E\,|v^{x_1,y}(s)-v^{x_2,y}(s)|_H\,ds\r).}
\end{array}\]
Then, thanks to Lemma \ref{rho}  we can conclude that $S:H\to
\mathcal{L}(L^\infty(D),H)$ is Lipschitz-continuous.

This implies that $\bar{G}:H\to \mathcal{L}(L^\infty(D),H)$ is
Lipschitz-continuous as well. Actually, thanks to \eqref{commg}
and to the fact that $\bar{G}(x_1)+\bar{G}(x_2)$ is invertible,
 for any $h \in\,L^\infty(D)$ and $k \in\,H$
\begin{equation}
\label{stima17}
\le<\le[\bar{G}(x_1)-\bar{G}(x_2)\r]h,k\r>_H=\le<\le[S(x_1)-S(x_2)\r]h,\le[\bar{G}(x_1)+\bar{G}(x_2)\r]^{-1}k\r>_H.
\end{equation}
Then, as
\begin{equation}
\label{stima18}
\|\le[\bar{G}(x_1)+\bar{G}(x_2)\r]^{-1}\|_{\mathcal{L}(H)}\leq \frac 1{2\d},
\end{equation}
we obtain
\[\le|\le<\le[\bar{G}(x_1)-\bar{G}(x_2)\r]h,k\r>_H\r|\leq
c\,|x_1-x_2|_H\,|h|_\infty \frac 1{2\d}\,|k|_H,\] and this
implies the Lipschitz-continuity of $\bar{G}:H\to
\mathcal{L}(L^\infty(D),H)$.

\bigskip

We conclude by showing that the operator $\bar{G}$ introduced in
Hypothesis \ref{H5} satisfies a suitable Hilbert-Schmidt property
which assures the well-posedness of the stochastic convolution
\[\int_0^t e^{(t-s)A_1}\bar{G}(u(s))dw^{Q_1}(s),\ \ \ \ \
t\geq 0,\] in $L^p(\Omega;C([0,T];H))$, for any $p\geq 1$ and
$T>0$ and for any process $u \in\,C([0,T];L^p(\Omega;H))$.

\begin{Lemma}
\label{hs} Assume Hypotheses \ref{H1}, \ref{H2} and \ref{H5}.
Then, for any $t>0$ and $x_1, x_2 \in\,H$ we have
\[\|e^{t A_1}\le[\bar{G}(x_1)-\bar{G}(x_2)\r]Q_1\|_2\leq
c(t)\,|x_1-x_2|_H\,t^{-\frac{\beta_1(\rho_1-2)}{2\rho_1}},\] for
some continuous increasing function $c(t)$.
\end{Lemma}

\begin{proof}
According to Hypothesis \ref{H1}, we have
\[\begin{array}{l}
\ds{\|e^{t
A_1}\le[\bar{G}(x_1)-\bar{G}(x_2)\r]Q_1\|_2^2=\sum_{k=1}^\infty
\le|e^{t
A_1}\le[\bar{G}(x_1)-\bar{G}(x_2)\r]Q_1 e_{1,k}\r|_H^2}\\
\vs \ds{\leq
\le(\sum_{k=1}^\infty\la_{1,k}^{\rho_1}\,|e_{1,k}|_\infty^2\r)^{\frac
2{\rho_1}}\le(\sum_{k=1}^\infty |e_{1,k}|_\infty^{-\frac
4{\rho_1-2}}\,\le|e^{t A_1}\le[\bar{G}(x_1)-\bar{G}(x_2)\r]
e_{1,k}\r|_H^{\frac{2\rho_1}{\rho_1-2}}\r)^{\frac {\rho_1-2}{\rho_1}}}\\
\vs \ds{\leq c\,\sup_{k \in\,\nat} |e_{1,k}|_\infty^{-\frac
4{\rho_1}}\le|e^{t A_1}\le[\bar{G}(x_1)-\bar{G}(x_2)\r]
e_{1,k}\r|_H^{\frac 4{\rho_1}}\le(\sum_{k=1}^\infty \le|e^{t
A_1}\le[\bar{G}(x_1)-\bar{G}(x_2)\r] e_{1,k}\r|_H^{2}\r)^{\frac
{\rho_1-2}{\rho_1}},}
\end{array}\]
and then, as $\bar{G}:H\to \mathcal{L}(L^\infty(D),H)$ is
Lipschitz-continuous, we conclude
\begin{equation}
\label{sola200} \|e^{t
A_1}\le[\bar{G}(x_1)-\bar{G}(x_2)\r]Q_1\|_2^2\leq
c(t)\,|x_1-x_2|_H^{\frac 4{\rho_1}}\le(\sum_{k=1}^\infty \le|e^{t
A_1}\le[\bar{G}(x_1)-\bar{G}(x_2)\r] e_{1,k}\r|_H^{2}\r)^{\frac
{\rho_1-2}{\rho_1}},\end{equation}
 for some continuous increasing
function $c(t)$.

Now, by using again the Lipschitz-continuity of $\bar{G}:H\to
\mathcal{L}(L^\infty(D),H)$, we have
\[\begin{array}{l}
\ds{\sum_{k=1}^\infty \le|e^{t
A_1}\le[\bar{G}(x_1)-\bar{G}(x_2)\r]
e_{1,k}\r|_H^{2}=\sum_{k=1}^\infty \sum_{h=1}^\infty \le|\le<e^{t
A_1}\le[\bar{G}(x_1)-\bar{G}(x_2)\r]
e_{1,k},e_{1,h}\r>_H\r|^{2}}\\
\vs \ds{=\sum_{h=1}^\infty \le|\le[\bar{G}(x_1)-\bar{G}(x_2)\r]
e_{1,h}\r|_H^{2} e^{-2t \a_{1,h}}\leq c\,|x_1-x_2|_H^2
\sum_{h=1}^\infty e^{-2t \a_{1,h}} |e_{1,h}|_H^{2}}
\end{array}\]
and then, if we take $\beta_1$ as in Hypothesis \ref{H1}, we
obtain
\[\sum_{k=1}^\infty \le|e^{t
A_1}\le[\bar{G}(x_1)-\bar{G}(x_2)\r] e_{1,k}\r|_H^{2}\leq
c\,|x_1-x_2|_H^2 t^{-\beta_1}\sum_{h=1}^\infty \a_{1,h}^{-\beta_1}
|e_{1,h}|_H^{2}\leq c\,|x_1-x_2|_H^2 t^{-\beta_1}.\] Thanks to
\eqref{sola200}, this implies our thesis.

\end{proof}

\subsection{Martingale problem and mild solution of the averaged equation} \label{4.2}

Since both the mapping $\bar{B}:H\to H$ and the mapping
$\bar{G}:H\to \mathcal{L}(L^\infty(D);H)$ are Lipschitz-continuous
and  Lemma \ref{hs} holds,   for any initial datum $x \in\,H$ the
averaged equation
\begin{equation}
\label{averaged} du(t)=\le[A_1
u(t)+\bar{B}(u(t))\r]\,dt+\bar{G}(u(t))\,dw^{Q_1}(t),\ \ \ \
u(0)=x,
\end{equation}
admits a unique mild solution $\bar{u}$ in
$L^p(\Omega,C([0,T];H))$, for any $p\geq 1$ and $T>0$  (for a proof see for example
\cite[Section 3]{cerrai1}). This means
that there exists a unique adapted process $\bar{u}
\in\,L^p(\Omega,C([0,T];H))$ such that for any $t\leq T$
\[\bar{u}(t)=e^{t
A_1}x+\int_0^t e^{(t-s)A_1}\bar{B}(\bar{u}(s))\,ds+\int_0^t
e^{(t-s)A_1}\bar{G}(\bar{u}(s))\,dw^{Q_1}(s),\] or, equivalently,
\[\le<\bar{u}(t),h\r>_H=\le<x,h\r>_H+\int_0^t
\le[\le<\bar{u}(s),A_1
h\r>_H+\le<\bar{B}(\bar{u}(s)),h\r>_H\r]\,ds+\int_0^t
\le<\bar{G}(\bar{u}(s))dw^{Q_1}(s),h\r>_H,\] for any $h
\in\,D(A_1)$.

Now, we recall the notion of martingale problem with parameters
$(x,A_1,\bar{B},\bar{G},Q_1)$. For any fixed $x \in\,H$, we denote
by $C_x([0,T];H)$ the space of continuous functions
$\omega:[0,T]\to H$ such that $\omega(0)=x$ and we denote by
$\eta(t)$ the canonical process on $C_x([0,T];H)$, which is
defined by
\[\eta(t)(\omega)=\omega(t),\ \ \ \ \ t \in\,[0,T].\]
Moreover, we denote by $\mathcal{E}_t$ the canonical filtration
$\sigma(\eta(s),\ s\leq t)$, for $t \in\,[0,T]$,  and by
$\mathcal{E}$ the canonical $\si$-algebra $\sigma(\eta(s),\ s\leq
T)$.

\begin{Definition}
A function $\varphi:H\to \reals$ is a {\em regular cylindrical
function} associated with the operator $A_1$ if there exist $k
\in\, \nat$, $f \in\,C_c^\infty(\mathbb{R}^k)$,
 $a_1,\ldots,a_k \in\,H$ and $N \in\,\nat$ such that
\[\varphi(x)=f(\le<x,P_N a_1\r>_H,\ldots,\le<x,P_N a_k\r>_H),\ \ \ \ x
\in\,H,\] where $P_N$ is the projection of $H$ onto $\text{{\em
span}}\le<e_{1,1},\ldots,e_{1,N}\r>$ and $\{e_{1,n}\}_{n
\in\,\nat}$ is the orthonormal basis diagonalizing $A_1$ and 
introduced in Hypothesis \ref{H1}.
\end{Definition}

In what follows we shall denote the set of all regular cylindrical
functions  by $\mathcal{R}(H)$. For any $\varphi
\in\,\mathcal{R}(H)$ and $x \in\,H$ we define
\begin{equation}
\label{av}
\begin{array}{l}
\ds{\mathcal{L}_{av}\,\varphi(x):=\frac 12
\text{Tr}\,\le[(\bar{G}(x)Q_1)^{\star}D^2\varphi(x)\bar{G}(x)Q_1\r]+\le<A_1
D\varphi(x),x\r>_H+\le<D\varphi(x),\bar{B}(x)\r>_H}\\
\vs \ds{=\frac 12\sum_{i,j=1}^kD^2_{ij} f(\le<x,P_N a_1\r>_H,\ldots,\le<x,P_N
a_k\r>_H)\le<\bar{G}(x) Q_1P_Na_i,\bar{G}(x) Q_1P_Na_j\r>_H}\\
\vs \ds{+\sum_{i=1}^kD_i f(\le<x,P_N
a_1\r>_H,\ldots,\le<x,P_N a_k\r>_H)\le(\le<x,A_1 P_N
a_i\r>_H+\le<\bar{B}(x),P_N a_i\r>_H\r).}
\end{array}\end{equation}
$\mathcal{L}_{av}$ is the Kolmogorov operator associated with the
averaged equation \eqref{averaged}. Notice that the expression
above is  meaningful, as for any $i=1,\ldots,k$ \[Q_1 P_N
a_i=\sum_{k=1}^N\la_{1,k}\le<a_i,e_{1,k}\r>_H e_{1,k}
\in\,L^{\infty}(D),\] and
\[A_1 P_N a_i=-\sum_{k=1}^N\a_{1,k}\le<a_i,e_{1,k}\r>_H e_{1,k}
\in\,H.\]

\begin{Definition}
A probability measure $\mathbb{Q}$ on $(C_x([0,T];H),\mathcal{E})$
is a solution of the martingale problem with parameters
$(x,A_1,\bar{B},\bar{G},Q_1)$ if the process
\[\varphi(\eta(t))-\int_0^t\mathcal{L}_{\text{av}}\,\varphi(\eta(s))\,ds,\
\ \ \ t \in\,[0,T],\] is an $\mathcal{E}_t$-martingale on
$(C_x([0,T];H),\mathcal{E},\mathbb{Q})$, for any $\varphi
\in\,\mathcal{R}(H)$.

\end{Definition}

As the coefficients $\bar{B}$ and $\bar{G}$ are
Lipschitz-continuous, the solution $\mathbb{Q}$ to the martingale
problem with parameters $(x,A_1,\bar{B},\bar{G},Q_1)$ exists, is
unique and coincides with $\mathcal{L}(\bar{u})$ (to this purpose
see \cite[Chapter 8]{dpz1} and also \cite[Theorem 5.9 and Theorem
5.10]{tz}).

\section{A priori bounds for the solution of system \eqref{eq0}}
\label{sec4}

In the present section we prove uniform estimates, with respect to
$\e \in\,(0,1]$, for the solution $u_\e$ of the slow motion
equation and for the solution  $v_\e$ of the fast  motion equation
in system \eqref{eq0}. As a consequence,  we will obtain the
tightness of the family $\{\mathcal{L}(u_\e)\}_{\e \in\,(0,1]}$ in
$C([0,T];H)$, for any $T>0$.

In what follows, for the sake of simplicity, we denote by
$|\cdot|_\theta$ the norm $|\cdot|_{D((-A_1)^\theta)}$. Moreover,
for any $\e>0$   we denote
\begin{equation}
\label{gammau} \Gamma_{1,\e}(t):=\int_0^t
e^{(t-s)A_1}G_1(u_\e(s),v_\e(s))\,dw^{Q_1}(s),\ \ \ \ \ t\geq 0.
\end{equation}

\begin{Lemma}
\label{lem1} Under Hypotheses \ref{H1} and \ref{H2},
there exists $\bar{\theta}>0$ and $\bar{p}\geq 1$ such that for
any $\e>0$, $T>0$, $p>\bar{p}$ and $\theta \in\,[0,\bar{\theta}]$
\begin{equation}
\label{sola35} \E\sup_{t\leq T}\,|\Gamma_{1,\e}(t)|_{\theta}^p\leq
c_{T,p,\theta}\int_0^T\le(1+\E\,|u_\e(r)|_H^p+\E\,|v_\e(r)|_H^p\r)\,ds,
\end{equation}
for some positive constant $c_{T,p,\theta}$ which is independent of $\e>0$.
\end{Lemma}

\begin{proof}
By using a factorization argument, for any $\a \in\,(0,1/2)$ we
have
\[\Gamma_{1,\e}(t)=c_\a\int_0^t (t-s)^{\a-1}
e^{(t-s)A_1} Y_{\e,\a}(s)\,ds,\] where
\[Y_{\e,\a}(s):=\int_0^s
(s-r)^{-\a}e^{(s-r)A_1}G_{1,\e}(s)\,dw^{Q_1}(r)\]
and \[G_{1,\e}(s):=G_1(u_\e(s),v_\e(s)).\]
 For any
$p>1/\a$ and $\theta>0$, we have
\begin{equation}
\label{sola40}
 \sup_{s\leq t}\,|\Gamma_{1,\e}(s)|_\theta^p\leq
c_{\a,p}\le(\int_0^ts^{(\a-1)\frac{p}{p-1}}\,ds\r)^{p-1}\int_0^t|Y_{\e,\a}(s)|_\theta^p\,ds
= c_{\a,p} t^{\a p-1}\int_0^t|Y_{\e,\a}(s)|_\theta^p\,ds.
\end{equation}
According to the Burkholder-Davis-Gundy inequality we have
\[\begin{array}{l}
\ds{\E\int_0^t|Y_{\e,\a}(s)|_\theta^p\,ds\leq
c_p\int_0^t\E\le(\int_0^s(s-r)^{-2\a}\|(-A_1)^\theta
e^{(s-r)A_1}G_{1,\e}(s)Q_1\|_{2}^2\,dr\r)^{\frac p2}\,ds.}
\end{array}\]
By  the same arguments as those used in the proof of Lemma \ref{rho}, we
have
\[\begin{array}{l}
\ds{\|(-A_1)^\theta e^{(s-r)A_1}G_{1,\e}(s)Q_1\|_{2}^2\leq
\sup_{k \in\,\nat}\,|(-A_1)^\theta e^{(s-r)A_1} G_{1,\e}(s)
e_{1,k}|_H^{\frac 4{\rho_1}}|e_{1,k}|_H^{-\frac
4{\rho_1}}}\\
\vs \ds{\times\  \kappa_1^{\frac 2{\rho_1}}\le(\sum_{k=1}^\infty
\le|(-A_1)^\theta e^{(s-r)A_1} G_{1,\e}(s)
e_{1,k}\r|_H^2\r)^{\frac{\rho_1-2}{\rho_1}}}\\
\vs \ds{ \leq c_\theta\,\kappa_1^{\frac 2{\rho_1}}(s-r)^{-\frac{4
\theta}{\rho_1}}\|G_{1,\e}(s)\|_{\mathcal{L}(L^\infty(D),H)}^{\frac
4{\rho_1}}\le(\sum_{k=1}^\infty \le|(-A_1)^\theta e^{(s-r)A_1}
G_{1,\e}(s) e_{1,k}\r|_H^2\r)^{\frac{\rho_1-2}{\rho_1}}.}
\end{array}\]
By proceeding again as in the proof of Lemma \ref{rho} we have
\[\sum_{k=1}^\infty \le|(-A_1)^\theta e^{(s-r)A_1}
G_{1,\e}(s) e_{1,k}\r|_H^2\leq
\|G_{1,\e}(s)\|_{\mathcal{L}(L^\infty(D),H)}^2\sum_{k=1}^\infty
|e_{1,k}|_0^2 \a_{1,k}^\theta e^{-\a_{1,k}(s-r)},\] and then,
thanks to \eqref{beta},  we get
\[\begin{array}{l}
\ds{\sum_{k=1}^\infty \le|(-A_1)^\theta e^{(s-r)A_1} G_{1,\e}(s)
e_{1,k}\r|_H^2\leq
\le(\frac{\beta_1+\theta}e\r)^{\beta_1+\theta}\zeta_1(s-r)^{-(\beta_1+\theta)}
\|G_{1,\e}(s)\|_{\mathcal{L}(L^\infty(D),H)}^2.} \end{array}\]
Therefore, if we fix $\bar{\theta}>0$ such that
\[\frac{\beta_1(\rho_1-2)+ \bar{\theta}(\rho_1+2)}{\rho_1}<1,\]
and if we set
\[K_{1,\theta}:=c_\theta\,\le(\frac{\beta_1+\theta}e\r)^{\frac{(\beta_1+\theta)
(\rho_1-2)}{\rho_1}}\zeta_1^{\frac{\rho_1-2}{\rho_1}}\kappa_1^{\frac
2{\rho_1}},\]
 for any $\theta \in\,[0,\bar{\theta}]$ we have
\[\begin{array}{l}
\ds{\E\int_0^t|Y_{\e,\a}(s)|_H^p\,ds}\\
\vs \ds{\leq c_p\,K_{1,\theta}^{\frac p2}\int_0^t\E\le(\int_0^s
\,(s-r)^{-\le(2\a+\frac{\beta_1(\rho_1-2)+
\theta(\rho_1+2)}{\rho_1}\r)}
\|G_{1,\e}(r)\|_{\mathcal{L}(L^\infty(D),H)}^2\,dr\r)^{\frac
p2}\,ds.} \end{array}\] Hence, if we  choose $\bar{\a}>0$  such
that
\[2\bar{\a}+\frac{\beta_1(\rho_1-2)+
\bar{\theta}(\rho_1+2)}{\rho_1}<1\] and  $p> \bar{p}:=1/\bar{\a}$,
by the Young inequality  this yields for $t \in\,[0,T]$
\[\begin{array}{l}
\ds{\E\int_0^t|Y_{\e,\bar{\a}}(s)|_{\theta}^p\,ds}\\
\vs \ds{\leq c_p\,K_{1,\theta}^{\frac p2} \le(\int_0^t
s^{-\le(2\bar{\a}+\frac{\beta_1(\rho_1-2)+
\theta(\rho_1+2)}{\rho_1}\r)} \,ds\r)^{\frac p2}\E\int_0^t
\|G_{1,\e}(s)\|_{\mathcal{L}(L^\infty(D),H)}^p\,ds}\\
\vs \ds{\leq c_{T,p}\ \int_0^t \le(1+\E\,|u_\e(s)|_H^p+\E\,|v_\e(s)|_H^p\r)\,ds.}
\end{array}\]
Thanks to \eqref{sola40}, this implies \eqref{sola35}.

\end{proof}
 Now, we can prove the first a-priori bounds for the solution
 $u_\e$ of the slow motion equation and for the solution $v_\e$ of the fast
 motion equation in system \eqref{eq0}.

\begin{Proposition}
\label{prop1} Under Hypotheses \ref{H1} and \ref{H2}, for any $T>0$
and $p\geq 1$ there exists a positive constant $c(p,T)$ such that
for any $x,y \in\,H$ and $\e \in\,(0,1]$
\begin{equation}
\label{unifue} \E \sup_{t \in\,[0,T]}\,|u_\e(t)|_H^p\leq
c(p,T)(1+|x|_H^p+|y|_H^p),
\end{equation}
and
\begin{equation}
\label{unifve} \int_0^T\E\,|v_\e(t)|_H^p\,dt\leq
c(p,T)(1+|x|_H^p+|y|_H^p).
\end{equation}
Moreover,
\begin{equation}
\label{unifvebis} \sup_{t \in\,[0,T]}\,\E\,|v_\e(t)|_H^2\leq c_T
(1+|x|_H^2+|y|_H^2).
\end{equation}
\end{Proposition}

\begin{proof}
Let $\e>0$ and $x,y \in\,H$ be fixed once for all and let
$\Gamma_{1,\e}(t)$ be the process defined in \eqref{gammau}. If we
set $\La_{1,\e}(t):=u_\e(t)-\Gamma_{1,\e}(t)$, we have
\[\frac{d}{dt}\La_{1,\e}(t)=A_1
\La_{1,\e}(t)+B_1(\La_{1,\e}(t)+\Gamma_{1,\e}(t),v_\e(t)),\ \ \ \
\ \La_{1,\e}(0)=x,\]and then for any $p\geq 2$ we have
\[\begin{array}{l}
\ds{\frac 1p\frac d{dt}|\La_{1,\e}(t)|_H^p=\le<A_1
\La_{1,\e}(t),\La_{1,\e}(t)\r>_H\,|\La_{1,\e}(t)|_H^{p-2}}\\
\vs \ds{+\le<B_1(\La_{1,\e}(t)+\Gamma_{1,\e}(t),v_\e(t))-
B_1(\Gamma_{1,\e}(t),v_\e(t)),\La_{1,\e}(t)\r>_H\,\,|\La_{1,\e}(t)|_H^{p-2}}\\
\vs \ds{+
\le<B_1(\Gamma_{1,\e}(t),v_\e(t)),\La_{1,\e}(t)\r>_H\,\,|\La_{1,\e}(t)|_H^{p-2}\leq c_p\,|\La_{1,\e}(t)|_H^p+c_p|B_1(\Gamma_{1,\e}(t),v_\e(t))|_H^p}\\
\vs \ds{\leq
c_p\,|\La_{1,\e}(t)|_H^{p}+c_p\le(1+|\Gamma_{1,\e}(t)|_H^p+|v_\e(t)|_H^p\r).}
\end{array}\]
This implies that
\[|\La_{1,\e}(t)|_H^{p}\leq e^{c_p t}|x|_H^p+c_p\int_0^t
e^{c_p(t-s)}\le(1+|\Gamma_{1,\e}(s)|_H^p+|v_\e(s)|_H^p\r)\,ds,\]
so that, for any $t \in\,[0,T]$
\[\begin{array}{l}
\ds{|u_\e(t)|_H^p\leq c_p\,|\Gamma_{1,\e}(t)|_H^p+c_p\,e^{c_p
t}|x|_H^p+c_p\int_0^t
e^{c_p(t-s)}\le(1+|\Gamma_{1,\e}(s)|_H^p+|v_\e(s)|_H^p\r)\,ds}\\
\vs \ds{\leq c_{T,p}\le(1+|x|_H^p+\sup_{s\leq
t}|\Gamma_{1,\e}(s)|_H^p+\int_0^t |v_\e(s)|_H^p\,ds\r).}
\end{array}\]
According to \eqref{sola35} (with $\theta=0$), we obtain
\[\E\sup_{s\leq t} |u_\e(s)|_H^p\leq
c_{T,p}\le(1+|x|_H^p\r)+c_{T,p}\,\int_0^t
\E\,|v_\e(s)|_H^p\,ds+c_{T,p}\int_0^t \le(1+\E\,\sup_{r\leq
s}|u_\e(r)|_H^p\r)\,ds,\] and hence, by comparison,
\begin{equation}
\label{sola23} \E\sup_{s\leq t} |u_\e(s)|_H^p\leq
c_{T,p}\le(1+|x|_H^p+\int_0^t \E\,|v_\e(s)|_H^p\,ds\r).
\end{equation}
Now, we have to estimate
\[\int_0^t \E\,|v_\e(s)|_H^p\,ds.\]
If we define
\[\Gamma_{2,\e}(t):=\frac 1{\sqrt{\e}}\int_0^t e^{\frac{(t-s)}{\e}A_2}
G_2(u_\e(s),v_\e(s))\,dw^{Q_2}(s),\] and set
$\La_{2,\e}(t):=v_\e(t)-\Gamma_{2,\e}(t)$, we have
\[\frac d{dt}\La_{2,\e}(t)=\frac 1\e\le[A_2
\La_{2,\e}(t)+B_2(u_\e(t),\La_{2,\e}(t)+\Gamma_{2,\e}(t))\r],\ \ \
\ \ \La_{2,\e}(0)=y.\] Hence, as before, for any $p\geq 1$ we have
\[\begin{array}{l}
\ds{\frac 1p\frac d{dt}|\La_{2,\e}(t)|_H^p=\frac 1\e\le<A_2
\La_{2,\e}(t),\La_{2,\e}(t)\r>_H\,|\La_{2,\e}(t)|_H^{p-2}}\\
\vs \ds{+\frac 1\e\le<B_2(u_\e(t),\La_{2,\e}(t)+\Gamma_{2,\e}(t))-
B_1(u_\e(t),\Gamma_{2,\e}(t)),\La_{2,\e}(t)\r>_H\,\,|\La_{2,\e}(t)|_H^{p-2}}\\
\vs \ds{+ \frac
1\e\le<B_2(u_\e(t),\Gamma_{2,\e}(t)),\La_{2,\e}(t)\r>_H\,\,|\La_{2,\e}(t)|_H^{p-2}}\\
\vs \ds{ \leq
-\frac{\la-L_{b_2}}{2\e}\,|\La_{2,\e}(t)|_H^p+\frac{c_p}\e\le(1+|u_\e(t)|_H^p+|\Gamma_{2,\e}(t)|_H^p\r).}
\end{array}\]
By comparison this yields
\begin{equation}
 \label{sola64}
\begin{array}{l} \ds{|v_\e(t)|_H^p\leq c_p\,
|\La_{2,\e}(t)|_H^p+c_p\,
|\Gamma_{2,\e}(t)|_H^p}\\
\vs \ds{\leq
c_p\,e^{-\frac{p(\la-L_{b_2})}{2\e}t}|y|_H^p+\frac{c_p}\e\int_0^t
e^{-\frac{p(\la-L_{b_2})}{2\e}(t-s)}\le(1+|u_\e(s)|_H^p+|\Gamma_{2,\e}(s)|_H^p\r)\,ds+c_p\,|\Gamma_{2,\e}(t)|_H^p.}
\end{array}\end{equation}
Therefore, by integrating with respect to $t$, we easily obtain
\begin{equation}
\label{sola25}
\begin{array}{l} \ds{\int_0^t|v_\e(s)|_H^p\,ds\leq
c_p\,\le(\e\,|y|_H^p+\int_0^t|\Gamma_{2,\e}(s)|_H^p\,ds+\int_0^t|u_\e(s)|_H^p\,ds+1\r).}
\end{array}\end{equation}
According to the Burkholder-Davis-Gundy inequality and to
\eqref{sola20}, we have
\begin{equation}
\label{sola63}
\begin{array}{l} \ds{\E\,|\Gamma_{2,\e}(s)|_H^p\leq
c_p\,\e^{-\frac p2}\,\E\le(\int_0^s\|e^{\frac{(s-r)}{\e}A_2}
G_2(u_\e(r),v_\e(r))
Q_2\|_{2}^2\,dr\r)^{\frac p2}}\\
\vs \ds{\leq c_p\,K_2^{\frac p2}\,\e^{-\frac
p2}\,\E\le(\int_0^s\le(\frac{s-r}\e\r)^{-\frac{\beta_2(\rho_2-2)}{\rho_2}}
e^{-\frac{\la(\rho_2+2)}{\e \rho_2}
(s-r)}\,\|G_2(u_\e(r),v_\e(r))\|_{\mathcal{L}(L^\infty(D),H)}^2\,dr\r)^{\frac
p2}}\\
\vs \ds{\leq c_p\,K_2^{\frac p2}\,\e^{-\frac
p2}\,\E\le(\int_0^s\le(\frac{s-r}\e\r)^{-\frac{\beta_2(\rho_2-2)}{\rho_2}}
e^{-\frac{\la(\rho_2+2)}{\e \rho_2}
(s-r)}\le(1+|u_\e(r)|_H^2+|v_\e(r)|_H^{2\gamma}\r)\,dr\r)^{\frac
p2},}
\end{array}
\end{equation}
so that
\begin{equation}\label{sola63bis}
\int_0^t\E\,|\Gamma_{2,\e}(s)|_H^p\,ds\leq c_{p}\,\int_0^t
\le(1+\E\,|u_\e(s)|_H^p+\E\,|v_\e(s)|_H^{p\gamma}\r)\,ds.\end{equation}
Due to \eqref{sola25} this allows to conclude
\[\int_0^t\E\,|v_\e(s)|_H^p\,ds\leq
c_p\,\le(\e\,|y|_H^p+\int_0^t\le(1+\E\,|u_\e(s)|_H^p\r)\,ds+\int_0^t\E\,|v_\e(s)|_H^{p\gamma}\,ds+1\r),\]
and then, as $\gamma$ is assumed to be strictly less than $1$, if
$\e \in\,(0,1]$ and $t \in\,[0,T]$ we obtain
\[\int_0^t\E\,|v_\e(s)|_H^p\,ds\leq \frac 12
\int_0^t\E\,|v_\e(s)|_H^p\,ds+c_p\,|y|_H^p+
c_p\,\int_0^t\E\,|u_\e(s)|_H^p\,ds+c_{p,T}.\] This yields
\begin{equation}
\label{sola26} \int_0^t\E\,|v_\e(s)|_H^p\,ds\leq c_p\,|y|_H^p+
c_p\,\int_0^t\E\,\sup_{r\leq s}|u_\e(r)|_H^p\,ds+c_{p,T}.
\end{equation}
Hence, if we plug \eqref{sola26} into \eqref{sola23}, we get
\[\E\,\sup_{s\leq t} |u_\e(s)|_H^p\leq
c_{T,p}\le(1+|x|_H^p+|y|_H^p\r)+c_{T,p}\,\int_0^t\E\,\sup_{r\leq
s}|u_\e(r)|_H^p\,ds,\] and from the Gronwall lemma \eqref{unifue}
follows. Now, in view of estimates \eqref{unifue} and
\eqref{sola63bis}, from \eqref{sola25} we obtain \eqref{unifve}.

Finally, let us prove \eqref{unifvebis}. From \eqref{sola63} with
$p=2$ we get
\[\sup_{t\leq T}\E\,|\Gamma_{2,\e}(t)|_H^2\leq c_2\,|y|_H^2+c_2\le(1+\sup_{t \leq T}
\E\,|u_\e(t)|_H^2+\sup_{t \leq T} \E\,|v_\e(t)|_H^{2\gamma}\r),\]
and then, if we substitute in \eqref{sola64}, we obtain
\[\begin{array}{l}
\ds{\E\,|v_\e(t)|_H^2\leq c_2\le(1+|y|_H^2+\sup_{t \leq T}
\E\,|u_\e(t)|_H^2\r)+c_2\,\sup_{t \leq T}
\E\,|v_\e(t)|_H^{2\gamma}.}
\end{array}\]
As $\gamma<1$, for any $\eta>0$ we can fix $c_{\eta}>0$ such that
\[c_2\,\sup_{t \leq T}
\E\,|v_\e(t)|_H^{2\gamma}\leq \eta\,\sup_{t \leq T}
\E\,|v_\e(t)|_H^{2}+c_{\eta}.\] Therefore, if we take $\eta\leq
1/2$, we obtain
\[\frac 12 \sup_{t \leq T}
\E\,|v_\e(t)|_H^{2}\leq c_2\le(1+|y|_H^2+\sup_{t \leq T}
\E\,|u_\e(t)|_H^2\r),\] and \eqref{unifvebis} follows from
\eqref{unifue}.

\end{proof}

Next, we prove uniform bounds for $u_\e$ in
$L^\infty(0,T;D((-A_1)^{\a}))$, for some $\a>0$.

\begin{Proposition}
\label{prop2} Under Hypotheses \ref{H1} and \ref{H2}, there exists
$\bar{\a}>0$ such that for any $T>0$, $p\geq 1$, $x
\in\,D((-A_1)^\a)$, with $\a \in\,[0, \bar{\a}]$, and $y \in\,H$
\begin{equation}
\label{sola28} \sup_{\e \in\,(0,1]}\E\,\sup_{t \leq
T}|u_\e(t)|^p_{\a}\leq c_{T,\a,p}\le(1+|x|^p_\a+|y|^p_H\r),
\end{equation}
for some positive constant $c_{T,\a,p}$.
\end{Proposition}

\begin{proof}
Assume that $x \in\,D((-A_1)^\a)$, for some $\a\geq 0$. We have
\[u_\e(t)=e^{tA_1} x+\int_0^t e^{(t-s)A_1}
B_1(u_\e(s),v_\e(s))\,ds+\int_0^t e^{(t-s)A_1}G_1(u_\e(s),v_\e(s))
\,dw^{Q_1}(s).\] If $\a<1/2$, $t\leq T$ and $p\geq 2$
\[\begin{array}{l}
\ds{\le|\int_0^t e^{(t-s)A_1} B_1(u_\e(s),v_\e(s))\,ds\r|^p_\a\leq
c_{p,\a}\le(\int_0^t (t-s)^{-\a} |B_1(u_\e(s),v_\e(s))|_H\,ds\r)^p}\\
\vs \ds{\leq
c_{p,\a}\le(\int_0^t(t-s)^{-\a}\le(1+|u_\e(s)|_H+|v_\e(s)|_H\r)\,ds\r)^p}\\
\vs \ds{\leq c_{p,\a}\le(1+\sup_{s\leq
T}|u_\e(s)|^p_H\r)T^{(1-\a)p}+c_{p,\a}\le(\int_0^T s^{-2\a}\,ds\r)^{\frac
p2} \le(\int_0^T|v_\e(s)|^p_H\,ds\r) T^{\frac{p-2}2},}
\end{array}\]
so that, thanks to \eqref{unifue} and \eqref{unifve},
\begin{equation}
\label{sola29} \E\,\sup_{t\leq T}\le|\int_0^t e^{(t-s)A_1}
B_1(u_\e(s),v_\e(s))\,ds\r|^p_\a \leq
c_{T,\a,p}\le(1+|x|^p_H+|y|^p_H\r).
\end{equation}
Concerning the stochastic term $\Gamma_{1,\e}(t)$, due to Lemma
\ref{lem1} and to \eqref{unifue} there exists $\bar{\theta}>0$
such that for any $\a\leq \bar{\theta}$ and $p\geq 1$
\begin{equation}
\label{sola41} \E\,\sup_{t\leq
T}\le|\Gamma_{1,\e}(t)\r|^p_{\a}\leq
c_{T,\a,p}\le(1+|x|^p_H+|y|^p_H\r).\end{equation}
 Hence, if we choose
$\bar{\a}:=\bar{\theta}\wedge 1/2$, thanks to \eqref{sola29} and
\eqref{sola41}, for any $p\geq 2$ we have
\[\begin{array}{l}
\ds{\E\sup_{t\leq T}\,|u_\e(t)|^p_\a\leq \sup_{t\leq T}\,|e^{t
A_1}x|^p_\a+\E\,\sup_{t\leq T}\le|\int_0^t e^{(t-s)A_1}
B_1(u_\e(s),v_\e(s))\,ds\r|^p_\a}\\
\vs
\ds{+\E\,\sup_{t\leq
T}\le|\Gamma_{1,\e}(t)\r|^p_{\a}\leq c_{T,\a,p}\le(1+|x|^p_\a+|y|^p_H\r).} \end{array}\]

\end{proof}

Next, we prove uniform bounds for the increments of the mapping $t
\in\,[0,T]\mapsto u_\e(t) \in\,H$.

\begin{Proposition}
\label{prop3} Under Hypotheses \ref{H1} and \ref{H2}, for any $\a>0$
there exists $\beta(\a)>0$ such that for any $T>0$, $p\geq 2$,
$|h|\leq 1$, $x \in\,D((-A_1)^{\a})$ and $y \in\,H$ it holds
\begin{equation}
\label{sola30} \sup_{\e \in\,(0,1]}\E\,|u_\e(t+h)-u_\e(t)|_H^p\leq
c_{T,\a,p} |h|^{ \beta(\a)\,p}\le(|x|_{\a}^p+|y|^p_H+1\r),\ \ \ \
t \in\,(0,T].
\end{equation}
\end{Proposition}

\begin{proof}
Without any loss of generality we can assume $h\geq 0$. For any
$t,h\geq 0$ we have
\[\begin{array}{l}
\ds{u_\e(t+h)-u_\e(t)=\le(e^{h A_1}-I\r)u_\e(t)}\\
\vs \ds{+\int_t^{t+h}
e^{(t+h-s)A_1}\,B_1(u_\e(s),v_\e(s))\,ds+\int_t^{t+h}e^{(t+h-s)A_1}\,G_1(u_\e(s),v_\e(s))\,dw^{Q_1}(s).}
\end{array}\]
In view of \eqref{sola28}, if we  fix $\a \in\,(0,\bar{\a})$ and
$p\geq 1$ we have
\begin{equation}
\label{uno}
\begin{array}{l}
\ds{\E\,\le|\le(e^{h A_1}-I\r)u_\e(t)\r|^p_H\leq c_p\,h^{
\a p} \,\E\,|u_\e(t)|^p_{\a}\leq c_{T,\a,p}\,h^{\a p}\,\le(1+|x|_\a^p+|y|^p_H\r).}
\end{array}
\end{equation}
In view of \eqref{unifue} and \eqref{unifve}
\begin{equation}
\label{due}
\begin{array}{l}
\ds{\E\le|\int_t^{t+h}
e^{(t+h-s)A_1}\,B_1(u_\e(s),v_\e(s))\,ds\r|_H^p\leq c\,h^{p-1}
\int_t^{t+h} \le(1+\E\,|u_\e(s)|_H^p+\E\,|v_\e(s))|_H^p\r)\,ds}\\
\vs \ds{\leq c_T h^p\le(1+\sup_{s\leq T}\E\,|u_\e(s)|_H^p\r)+c\,
h^{p-1}\int_0^T\E\,|v_\e(s)|_H^p\,ds\leq
c_{T,p}\le(1+|x|_H^p+|y|_H^p\r)h^{p-1}.}
\end{array}
\end{equation}
Finally, for the stochastic term, by using \eqref{sola20}, for any
$t\leq T$ and $p\geq 1$ we have
\[\begin{array}{l}
\ds{\E\,\le|\int_t^{t+h}
e^{(t+h-s)A_1}\,G_1(u_\e(s),v_\e(s))\,dw^{Q_1}(s)\r|_H^p}\\
\vs
\ds{\leq c_p\,\E\le(\int_t^{t+h}\|e^{(t+h-s)A_1}\,G_1(u_\e(s),v_\e(s))Q_1\|_{2}^2\,ds\r)^{\frac p2}}\\
\vs \ds{\leq c_p\, K_1^{\frac p2}\E\le(\int_t^{t+h}
(t+h-s)^{-\frac{\beta_2(\rho_2-2)}{\rho_2}}
\|G_1(u_\e(s),v_\e(s))\|_{\mathcal{L}
(L^\infty(D),H)}^2\,ds\r)^{\frac p2}.}
\end{array}\]
Then, if we take $\bar{p}\geq 1$ such that
\[\frac{\beta_2(\rho_2-2)}{\rho_2}\frac{\bar{p}}{\bar{p}-2}<1,\]
for any $p\geq \bar{p}$ we have
\[\begin{array}{l}
\ds{\E\,\le|\int_t^{t+h}
e^{(t+h-s)A_1}\,G_1(u_\e(s),v_\e(s))\,dw^{Q_1}(s)\r|_H^p}\\
\vs
\ds{\leq c_{T,p}\,h^{\frac{p-2}2-\frac{\beta_2(\rho_2-2)}{\rho_2}\frac p2}\int_0^T\le(1+\E\,|u_\e(s)|_H^p+\E\,|v_\e(s)|_H^p\r)\,ds,}
\end{array}\]
and, thanks to \eqref{unifue} and \eqref{unifve}, we conclude
\begin{equation}
\label{tre} 
\begin{array}{l}
\ds{\E\,\le|\int_t^{t+h} e^{(t+h-s)A_1}\,G_1(u_\e(s),v_\e(s))\,dw^{Q_1}(s)\r|_H^p}\\
\vs
\ds{\leq
c_{T,p}\le(1+|x|_H^p+|y|_H^p\r)h^{\le(1-\frac 2{\bar{p}}-\frac{\beta_2(\rho_2-2)}{\rho_2}\r)\frac
p2}.}
\end{array}
\end{equation}
Therefore, collecting together \eqref{uno}, \eqref{due} and
\eqref{tre}, we obtain
\[\begin{array}{l}
\ds{\E\,|u_\e(t+h)-u_\e(t)|_H^p\leq
c_{T,\a,p}\,h^{\a p}\le(1+|x|^p_{\a}+|y|_H^p\r)}\\
\vs \ds{+c_{T,p}\le( h^{\le(1-\frac 2{\bar{p}}-\frac{\beta_2(\rho_2-2)}{\rho_2}\r)\frac
p2}+h^{p-1}\r)\le(1+|x|_H^p+|y|_H^p\r),}
\end{array}\] and, as we are assuming $|h|\leq 1$, \eqref{sola30} follows for any $p\geq \bar{p}$ by taking
\[\beta(\a):=\min\,\le\{\a, \frac 12\le(1-\frac 2{\bar{p}}-\frac{\beta_2(\rho_2-2)}{\rho_2}\r)\r\}.\]
Estimate \eqref{sola30} for $p<\bar{p}$ follows from the H\"older inequality.

\end{proof}

As a  consequence of Proposition \ref{prop2} and Proposition
\ref{prop3} we have the following fact.
\begin{Corollary}
\label{corollary1} Under Hypotheses \ref{H1} and \ref{H2}, for any
$T>0$, $x \in\,D((-A_1)^\a)$, with $\a>0$, and $y \in\,H$ the
family $\{\mathcal{L}(u_\e)\}_{\e \in\,(0,1]}$ is tight in
$C([0,T];H)$.
\end{Corollary}

\begin{proof}
Let $\a>0$ be fixed and let $x \in\,D((-A_1)^\a)$ and $y \in\,H$.
According to \eqref{sola30}, in view of the Garcia-Rademich-Rumsey
Theorem,  there exists $\bar{\beta}>0$ such that for any $p\geq 1$
\[ \sup_{\e
\in\,(0,1]}\E\,|u_\e|^p_{C^{\bar{\beta}}([0,T];H)}\leq
c_{T,p}\le(1+|x|^p_\a+|y|^p_H\r).
\]
Due to Proposition \ref{prop2}, this implies that for any $\eta>0$
we can find $R_\eta>0$ such that
\[\mathbb{P}\,\le(u_\e \in\,K_{R_\eta}\r)\geq 1-\eta,\ \ \ \ \ \
\e \in\,(0,1],\] where, by the Ascoli-Arzel\`a Theorem,
$K_{R_\eta}$ is the compact subset of $C([0,T];H)$ defined by
\[K_{R_\eta}:=\le\{\,u
\in\,C([0,T];H)\,:\,|u|_{C^{\bar{\beta}}([0,T];H)}+ \sup_{t
\in\,[0,T]}|u(t)|_\a\leq R_\eta\,\r\}.\] This implies that the
family of probability measures $\{\mathcal{L}(u_\e)\}_{\e
\in\,(0,1]}$ is tight in $C([0,T];H)$.
\end{proof}

We conclude this section by noticing that with arguments analogous to those used in the proof of Propositions \ref{prop1}, \ref{prop2} and \ref{prop3}, we can  obtain a priori bounds also for the
conditional second momenta of the $H$-norms of $u_\e$ and $v_\e$.

\begin{Proposition}
Assume Hypotheses \ref{H1} and \ref{H2}. Then, for any $0\leq
s<t\leq T$ and any $\e \in\,(0,1]$ the following facts holds.
\begin{enumerate}
\item There exists $\bar{\a}>0$ such that for any $x
\in\,D((-A_1)^\a)$, with $\a \in\,[0,\bar{\a}]$, and $y \in\,H$
\[
\E\le(|u_\e(t)|^2_{\a}\le|\mathcal{F}_s\r.\r)\leq
c_{T,\a}\le(1+|u_\e(s)|^2_\a+|v_\e(s)|^2_H\r),\ \ \ \
\mathbb{P}-\text{{\em a.s.},}
\]
for some constant $c_{T,\a}$ independent of $\e$.

\item For any $x,y \in\,H$
\begin{equation}
\label{sola29con} \E\le(|v_\e(t)|^2_{H}\le|\mathcal{F}_s\r.\r)\leq
c_{T}\le(1+|u_\e(s)|^2_H+|v_\e(s)|^2_H\r),\ \ \ \
\mathbb{P}-\text{{\em a.s.},}
\end{equation}
for some constant $c_{T}$ independent of $\e$.

\item For any $\a>0$
there exists $\beta(\a)>0$ such that for any $x
\in\,D((-A_1)^{\a})$ and $y \in\,H$
\[
\E\le(|u_\e(t)-u_\e(s)|_H^2\le|\mathcal{F}_s\r.\r)\leq c_{T,\a}
(t-s)^{2\, \beta(\a)}\le(|u_\e(s)|_{\a}^2+|v_\e(s)|^2_H+1\r),\ \ \
\ \mathbb{P}-\text{{\em a.s.}},
\]
for some constant $c_{T,\a}$ independent of $\e$.

\end{enumerate}

\end{Proposition}

\section{The key lemma}
\label{sec5}

We introduce the Kolmogorov operator associated with the
slow motion equation, with frozen fast component, by setting for
any $\varphi \in\,\mathcal{R}(H)$ and $x,y \in\,H$
\begin{equation}
\label{sl}\begin{array}{l}
\ds{\mathcal{L}_{sl}\,\varphi(x,y)}\\
\vs \ds{=\frac 12
\text{Tr}\,\le[Q_1G_1(x,y)D^2\varphi(x)G_1(x,y)Q_1\r]+\le<A_1
D\varphi(x),x\r>_H+\le<D\varphi(x),B_1(x,y)\r>_H}\\
\vs \ds{=\frac 12\sum_{i,j=1}^kD^2_{ij} f(\le<x,P_N a_1\r>_H,\ldots,\le<x,P_N
a_k\r>_H)\le<G_1(x,y) Q_{1,N} a_i,G_1(x,y) Q_{1,N} a_j\r>_H}\\
\vs \ds{+\sum_{i=1}^kD_i f(\le<x,P_N
a_1\r>_H,\ldots,\le<x,P_N a_k\r>_H)\le(\le<x,A_{1,N}
a_i\r>_H+\le<B_1(x,y),P_N a_i\r>_H\r).}
\end{array}\end{equation}

\begin{Lemma}
\label{lemma51}
Assume Hypotheses \ref{H1}-\ref{H5} and fix $x
\in\,D((-A_1)^\a)$, with $\a>0$, and $y \in\,H$. Then, for any
$\varphi \in\,\mathcal{R}(H)$ and  $0\leq t_1<t_2\leq T$
\begin{equation}
\label{sw3} \lim_{\e\to
0}\E\le|\int_{t_1}^{t_2}\E\le(\mathcal{L}_{sl}\,\varphi(u_{\e}(r),v_{\e}(r))-
\mathcal{L}_{\text{av}}\,\varphi(u_{\e}(r))\le|\mathcal{F}_{t_1}\r.\r)\,dr\r|=0.\end{equation}
\end{Lemma}

\begin{proof}
By using the Khasminskii idea introduced in \cite{khas},  we  realize a partition of  $[0,T]$ into
intervals of size $\d_\e>0$, to be chosen later on, and for
each $\e>0$ we denote by $\hat{v}_\e(t)$ the solution of the
problem
\begin{equation}
\label{hatve}
\begin{array}{l} \ds{\hat{v}_\e(t)=e^{(t-k
\d_\e)\frac{A_2}\e }v_\e(k \d_\e)+\frac 1\e\int_{k \d_\e}^t
e^{(t-s)\frac
{A_2}\e}B_{2}(u_\e(k \d_\e),\hat{v}_\e(s))\,ds}\\
\vs \ds{+\frac 1{\sqrt{\e}}\int_{k \d_\e}^t e^{(t-s)\frac {A_2}\e}
G_2(u_\e(k \d_\e),\hat{v}_\e(s))\,dw^{Q_2}(s),\ \ \ \ t \in\,[k
\d_\e,(k+1)\d_\e),}
\end{array}
\end{equation}
for $k=0,\ldots,[T/\d_\e]$. In what follows we shall set $\zeta_\e:=\d_\e/\e$.

\medskip

{\em Step 1.} 
Now, we prove that there exist $\kappa_1, \kappa_2>0$  such that if we
set
\[\zeta_\e=\le(\log \frac 1{\e^{\,\kappa_2}}\r)^{\kappa_1},\]
then
\begin{equation}
\label{sola52} \lim_{\e\to 0} \sup_{t
\in\,[0,T]}\,\E\,|\hat{v}_\e(t)-v_\e(t)|^2_H=0.
\end{equation}

If we fix $k=0,\ldots,[T/\d_\e]$ and take $t \in\,[k
\d_\e,(k+1)\d_\e)$, we have
\[\begin{array}{l} \ds{v_\e(t)=e^{(t-k
\d_\e)\frac{A_2}\e }v_\e(k \d_\e)+\frac 1\e\int_{k \d_\e}^t
e^{(t-s)\frac
{A_2}\e}B_{2}(u_\e(s),v_\e(s))\,ds}\\
\vs \ds{+\frac 1{\sqrt{\e}}\int_{k \d_\e}^t e^{(t-s)\frac {A_2}\e}
G_2(u_\e(s),v_\e(s))\,dw^{Q_2},}
\end{array}\]
so that
\[\begin{array}{l} \ds{\E\,\le|\hat{v}_\e(t)-v_\e(t)\r|_H^2\leq \frac {2\, \d_\e}{\e^2}\int_{k \d_\e}^t
\E\,\le|B_{2}(u_\e(k \d_\e),\hat{v}_\e(s))-B_{2}(u_\e(s),v_\e(s))\r|_H^2\,ds}\\
\vs \ds{+\frac 2{\e}\,\E\,\le|\int_{k \d_\e}^t e^{(t-s)\frac
{A_2}\e} \le[G_2(u_\e(k
\d_\e),\hat{v}_\e(s))-G_2(u_\e(s),v_\e(s))\r]\,dw^{Q_2}\r|_H^2.}
\end{array}\]
For the first term,  we have
\begin{equation}
\label{sola53}
\begin{array}{l} \ds{\frac{\d_\e}{\e^2}\int_{k \d_\e}^t
\E\,\le|B_{2}(u_\e(k
\d_\e),\hat{v}_\e(s))-B_{2}(u_\e(s),v_\e(s))\r|_H^2\,ds}\\
\vs \ds{\leq \frac{c}{\e}\int_{k \d_\e}^t\zeta_\e\,\le(\E\,|u_\e(k
\d_\e)-u_\e(s)|_H^2+\E\,|\hat{v}_\e(s)-v_\e(s)|_H^2\r)\,ds.}
\end{array}\end{equation}
For the second term, by proceeding as in the proof of Proposition
\ref{prop1} we obtain
\begin{equation}
\label{sola54}
\begin{array}{l} \ds{\E\,\le|\int_{k \d_\e}^t
e^{(t-s)\frac {A_2}\e} \le[G_2(u_\e(k
\d_\e),\hat{v}_\e(s))-G_2(u_\e(s),v_\e(s))\r]\,dw^{Q_2}\r|_H^2}\\
\vs \ds{\leq c\int_{k \d_\e}^t
\le(\frac{t-s}\e\r)^{-\frac{\beta_2(\rho_2-2)}{\rho_2}}e^{-\frac{\la(\rho_2+2)}{\e
\rho_2} (t-s)}\le(\E\,|u_\e(k
\d_\e)-u_\e(s)|_H^2+\E\,|\hat{v}_\e(s)-v_\e(s)|_H^2\r)\,ds.}
\end{array}\end{equation}
In view of \eqref{sola30}, we have
\[\begin{array}{l}
\ds{\frac 1\e\int_{k \d_\e}^t\le[
\le(\frac{t-s}\e\r)^{-\frac{\beta_2(\rho_2-2)}{\rho_2}}e^{-\frac{\la(\rho_2+2)}{\e
\rho_2} (t-s)}+\zeta_\e\r]\E\,|u_\e(k \d_\e)-u_\e(s)|_H^2\,ds}\\
\vs \ds{\leq \frac{c_T}{\e}\int_{k \d_\e}^t\le[
\le(\frac{t-s}\e\r)^{-\frac{\beta_2(\rho_2-2)}{\rho_2}}e^{-\frac{\la(\rho_2+2)}{\e
\rho_2} (t-s)}+\zeta_\e\r](s-k
\d_\e)^{2\beta(\a)}\,ds\le(1+|x|_\a^2+|y|_H^2\r)}\\
\vs \ds{\leq c_T\,
\d_\e^{2\beta(\a)}\le(1+\zeta_\e^2\r)\le(1+|x|_\a^2+|y|_H^2\r).}
\end{array}\]
Moreover \[\begin{array}{l} \ds{\frac 1\e\int_{k \d_\e}^t\le[
\le(\frac{t-s}\e\r)^{-\frac{\beta_2(\rho_2-2)}{\rho_2}}e^{-\frac{\la(\rho_2+2)}{\e
\rho_2}
(t-s)}+\zeta_\e\r]\E\,|\hat{v}_\e(s)-v_\e(s)|_H^2\,ds}\\
\vs \ds{\leq \frac c\e
\le(\e^{\frac{\beta_2(\rho_2-2)}{\rho_2}}+\zeta_\e\,\d_\e^{\frac{\beta_2(\rho_2-2)}{\rho_2}}\r)\int_{k
\d_\e}^t(t-s)^{-\frac{\beta_2(\rho_2-2)}{\rho_2}}
\E\,|\hat{v}_\e(s)-v_\e(s)|_H^2\,ds.}
\end{array}\] Then, thanks to \eqref{sola53} and
\eqref{sola54}, we obtain \begin{equation}
\label{sola56}
\begin{array}{l}
\ds{\E\,\le|\hat{v}_\e(t)-v_\e(t)\r|_H^2\leq c_T\,
\d_\e^{2\beta(\a)}\le(1+\zeta_\e^2\r)\le(1+|x|_\a^2+|y|_H^2\r)}\\
\vs \ds{+c\,
\e^{\frac{\beta_2(\rho_2-2)}{\rho_2}-1}\le(1+\zeta_\e^{1+\frac{\beta_2(\rho_2-2)}{\rho_2}}\r)
\int_{k \d_\e}^t(t-s)^{-\frac{\beta_2(\rho_2-2)}{\rho_2}}
\E\,|\hat{v}_\e(s)-v_\e(s)|_H^2\,ds.}
\end{array}\end{equation}

Now, we recall the following simple fact (for a proof see e.g.
\cite{fontbona}).
\begin{Lemma} \label{sola55}
If $M,L,\theta$ are positive constant and $g$ is a nonnegative
function such that
\[g(t)\leq M+L\int_{t_0}^t (t-s)^{\theta-1} g(s)\,ds,\ \ \ \ \ t\geq t_0,\]
then
\[g(t)\leq M+\frac{M
L}{\theta}(t-t_0)^\theta+L^2\int_0^1 r^{\theta-1}(1-r)^{\theta-1}\,dr\,\int_{t_0}^t(t-s)^{2\theta-1}g(s)\,ds,\
\ \ \ t\geq t_0.\]
\end{Lemma}
Notice that if we iterate the lemma above $n$-times, we find
\[g(t)\leq c_{1,n,\theta}\,
M\le(1+L^{2^n-1}(t-t_0)^{2^n-1}\r)+c_{2,n,\theta}\,
L^{2^n}\int_{t_0}^t (t-s)^{2^n \theta-1}g(s)\,ds,\] for some
positive  constants $c_{1,n,\theta}$ and $c_{2,n,\theta}$.

If we apply $\bar{n}$-times Lemma \ref{sola55} to \eqref{sola56},
with $\bar{n} \in\,\nat$ such that
\[2^{\bar{n}} \theta:=2^{\bar{n}}\le(1-\frac{\beta_2(\rho_2-2)}{\rho_2}\r)>1,\]
we get
\[\begin{array}{l}
\ds{\E\,\le|\hat{v}_\e(t)-v_\e(t)\r|_H^2}\\
\vs \ds{\leq c_T
\d_\e^{2\beta(\a)}\le(1+\zeta_\e^2\r)\le(1+|x|_\a^2+|y|_H^2\r)
\le(1+
\e^{-(2^{\bar{n}}-1)\theta}\le(1+\zeta_\e^{(2^{\bar{n}}-1)\le(2-\theta\r)}\r)\d_\e^{(2^{\bar{n}}-1)\theta}\r)}\\
\vs \ds{+c
\,\e^{-2^{\bar{n}}\theta}\le(1+\zeta_\e^{2^{\bar{n}}\le(2-\theta\r)}\r)\int_{k
\d_\e}^t(t-s)^{2^{\bar{n}}\theta-1}\E\,\le|\hat{v}_\e(s)-v_\e(s)\r|_H^2\,ds}\\
\vs \ds{\leq c_T \d_\e^{2\beta(\a)}\le(1+|x|_\a^2+|y|_H^2\r)
\le(1+ \zeta_\e^{2^{\bar{n}+1}}\r)+\frac c{\d_\e}
\,\le(1+\zeta_\e^{2^{\bar{n}+1}}\r)\int_{k
\d_\e}^t\E\,\le|\hat{v}_\e(s)-v_\e(s)\r|_H^2\,ds.}
\end{array}\]
From the Gronwall Lemma this yields
\[\E\,\le|\hat{v}_\e(t)-v_\e(t)\r|_H^2\leq
c_T \d_\e^{2\beta(\a)}\le(1+|x|_\a^2+|y|_H^2\r) \le(1+
\zeta_\e^{2^{\bar{n}+1}}\r)\exp\le(c\,\zeta_\e^{2^{\bar{n}+1}}\r).\]
Now, since we have
\[\exp\le(c\,\zeta_\e^{2^{\bar{n}+1}}\r)=\exp\le(c\le(\log \frac
1{\e^{\kappa_2}}\r)^{\kappa_1\,2^{\bar{n}+1}}\r),\] if we take
$\kappa_1:=2^{-(\bar{n}+1)}$ and $\kappa_2<2\beta(\a) c^{-1}$, we
conclude that \eqref{sola52} holds.

Moreover, as for $t \in\,[k\d_\e,(k+1)\d_\e]$ the process
$\hat{v}_\e(t)$ is the mild solution of the problem
\[dv(t)=\frac 1\e\,\le[A_2
v(t)+B_2(u_\e(k\d_\e),v(t))\r]\,dt+\frac
1{\sqrt{\e}}\,G_2(u_\e(k\d_\e),v(t))\,dw^{Q_2}(t),\ \ \ \
v(k\d_\e)=v_\e(k\d_\e),\]  with the same arguments as those used to prove
\eqref{unifvebis} and \eqref{sola29con} we obtain
\begin{equation}
\label{sw44} \sup_{t \in\,[0,T]}\E\,|\hat{v}_\e(t)|_H^2\leq
c_T\le(1+|x|_H^2+|y|_H^2\r)
\end{equation}
and, for any $t \in\,[k\d_\e,(k+1)\d_\e]$,
\[
\E\le(|\hat{v}_\e(t)|_H^2\le|\mathcal{F}_{k\d_\e}\r.\r)\leq
c\le(1+|u_\e(k\d_\e)|_H^2+|v_\e(k\d_\e)|_H^2\r),\ \
\mathbb{P}-\text{a.s.}
\]

\medskip

{\em Step 2.}
Now, we fix $\varphi \in\,\mathcal{R}(H)$. We can assume that
\[\varphi(x)=f(\le<x,P_N a_1\r>_H,\ldots,\le<x,P_N a_k\r>_H),\]
 for
some $f \in\,C^\infty_c(\mathbb{R}^k)$ and $k,N \in\,\nat$.
According to \eqref{av} and \eqref{sl}, we have
\[\begin{array}{l}
\ds{\mathcal{L}_{sl}\,\varphi(u_{\e}(r),v_{\e}(r))-
\mathcal{L}_{av}\,\varphi(u_{\e}(r))=\frac 12 \sum_{i,j=1}^k
I^{\e}_{ij}(r)+\sum_{i=1}^k J^{\e}_i(r),}
\end{array}\] where
\[\begin{array}{l}
\ds{I^{\e}_{ij}:=D^2_{ij} f(\le<u_{\e},P_N a_1\r>_H,\ldots,\le<u_{\e},P_N
a_k\r>_H)}\\
\vs \ds{\le(\le<G_1(u_{\e},v_{\e}) Q_{1,N}a_i,G_1(u_{\e},v_{\e})
Q_{1,N}a_j\r>_H\!-\!\le<\bar{G}(u_{\e}) Q_{1,N}a_i,\bar{G}(u_{\e})
Q_{1,N}a_j\r>_H \r)}
\end{array}\]
and
\[\begin{array}{l}
\ds{J^{\e}_i=D_i f(\le<u_{\e},P_N
a_1\r>_H,\ldots,\le<u_{\e},P_N
a_k\r>_H)\le<B_1(u_{\e},v_{\e})-\bar{B}(u_{\e}),P_N
a_i\r>_H.}
\end{array}\]
Hence, if we prove that for any $i,j=1,\ldots,k$
\begin{equation}
\label{sw1} \lim_{\e\to
0}\E\le|\int_{t_1}^{t_2}\E\le(I^{\e}_{ij}(r)\le|\mathcal{F}_{t_1}\r.\r)\,dr\r|=0
\end{equation}
and
\begin{equation}
\label{sw2} \lim_{\e\to
0}\E\le|\int_{t_1}^{t_2}\E\le(J^{\e}_{i}(r)\le|\mathcal{F}_{t_1}\r.\r)\,dr\r|=0,
\end{equation}
 we immediately get \eqref{sw3}.

We have
\[\begin{array}{l}
\ds{\int_{t_1}^{t_2}\E\le(I^{\e}_{ij}(r)\le|\mathcal{F}_{t_1}\r.\r)\,dr=
\sum_{l=1}^3\int_{t_1}^{t_2}\E\le(I^{\e}_{l,ij}(r)
\le|\mathcal{F}_{t_1}\r.\r)\,dr,}
\end{array}\]
where
\[\begin{array}{l}
\ds{I^{\e}_{1,ij}(r):= D^2_{ij} f(\le<u_{\e}(r),P_N a_1\r>_H,\ldots,\le<u_{\e}(r),P_N
a_k\r>_H)}\\
\vs \ds{\le<G_1(u_{\e}(r),v_{\e}(r))
Q_{1,N}a_i,G_1(u_{\e}(r),v_{\e}(r)) Q_{1,N}a_j\r>_H}\\
\vs \ds{- D^2_{ij} f(\le<u_{\e}([r/\d_\e]\d_\e),P_N
a_1\r>_H,\ldots,\le<u_{\e}([r/\d_\e]\d_\e),P_N
a_k\r>_H)}\\
\vs \ds{\le<G_1(u_{\e}([r/\d_\e]\d_\e),\hat{v}_{\e}(r))
Q_{1,N}a_i,G_1(u_{\e}([r/\d_\e]\d_\e),\hat{v}_{\e}(r))
Q_{1,N}a_j\r>_H,}
\end{array}\]
and
\[\begin{array}{l}
\ds{I^{\e}_{2,ij}(r):= D^2_{ij} f(\le<u_{\e}([r/\d_\e]\d_\e),P_N
a_1\r>_H,\ldots,\le<u_{\e}([r/\d_\e]\d_\e),P_N
a_k\r>_H)}\\
\vs \ds{\le(\le<G_1(u_{\e}([r/\d_\e]\d_\e),\hat{v}_{\e}(r))
Q_{1,N}a_i,G_1(u_{\e}([r/\d_\e]\d_\e),\hat{v}_{\e}(r))
Q_{1,N}a_j\r>_H\r.}\\
\vs{\le.-\le<\bar{G}(u_{\e}([r/\d_\e]\d_\e))
Q_{1,N}a_i,\bar{G}(u_{\e}([r/\d_\e]\d_\e)) Q_{1,N}a_j\r>_H\r),}
\end{array}\]
and
\[\begin{array}{l}
\ds{I^{\e}_{3,ij}(r):= D^2_{ij} f (\le<u_{\e}([r/\d_\e]\d_\e),P_N
a_1\r>_H,\ldots,\le<u_{\e}([r/\d_\e]\d_\e),P_N
a_k\r>_H)}\\
\vs \ds{\le<\bar{G}(u_{\e}([r/\d_\e]\d_\e))
Q_{1,N}a_i,\bar{G}(u_{\e}([r/\d_\e]\d_\e)) Q_{1,N}a_j\r>_H}\\
\vs \ds{- D^2_{ij} f (\le<u_{\e}(r),P_N a_1\r>_H,\ldots,\le<u_{\e}(r),P_N
a_k\r>_H)\le<\bar{G}(u_{\e}(r)) Q_{1,N}a_i,\bar{G}(u_{\e}(r))
Q_{1,N}a_j\r>_H.}
\end{array}\]
It is immediate to check that
\[\begin{array}{l}
\ds{|I^{\e}_{1,ij}(r)|+|I^{\e}_{3,ij}(r)|\leq
c\,\le(|u_\e([r/\d_\e]\d_\e)-u_\e(r)|_H+|v_\e(r)-\hat{v}_\e(r)|_H\r)}\\
\vs
\ds{\le(1+|u_\e(r)|_H^2+|v_\e(r)|_H^2+|u_\e([r/\d_\e]\d_\e)|_H+|\hat{v}_\e(r)|_H\r),}
\end{array}\]
so that
\[\begin{array}{l}
\ds{\le(\E\int_{t_1}^{t_2}\le[|I^{\e}_{1,ij}(r)|+|I^{\e}_{3,ij}(r)|\r]\,dr\r)^2}\\
\vs \ds{\leq c\int_{t_1}^{t_2} \le[\E
\,|u_\e([r/\d_\e]\d_\e)-u_\e(r)|_H^2+\E\,|v_\e(r)-\hat{v}_\e(r)|^2_H\r]\,dr}\\
\vs \ds{\times \int_{t_1}^{t_2} \le[
1+\E\,|u_\e(r)|_H^4+\E\,|v_\e(r)|_H^4+\E\,|u_\e([r/\d_\e]\d_\e)|_H^2+\E\,|\hat{v}_\e(r)|_H^2\r]\,dr.}
\end{array}\]
According to \eqref{sola30}, \eqref{unifue}, \eqref{unifve} and
\eqref{sw44}, we conclude
\[\begin{array}{l}
\ds{\le(\E\int_{t_1}^{t_2}\le[|I^{\e}_{1,ij}(r)|+|I^{\e}_{3,ij}(r)|\r]\,dr\r)^2}\\
\vs \ds{\leq c_T\le(\d_\e^{2\beta(\a)}+\sup_{t
\in\,[0,T]}\E\,|v_\e(t)-\hat{v}_\e(t)|^2_H\r)\le(1+|x|_H^4+|y|_H^4+|x|_\a^2\r),}
\end{array}\]
so that, due to \eqref{sola52}
\begin{equation}
\label{i13} \lim_{\e\to
0}\E\int_{t_1}^{t_2}\le[|I^{\e}_{1,ij}(r)|+|I^{\e}_{3,ij}(r)|\r]\,dr=0.
\end{equation}

\medskip

Next, let us estimate $I_{2,ij}$. We have
\[\begin{array}{l}
\ds{\int_{t_1}^{t_2}
\E\,\le(I_{2,ij}(r)\le|\mathcal{F}_{t_1}\r)\r.\,dr=\sum_{k=[t_1/\d_\e]+1}^{[t_2/\d_\e]-1}
\int_{k\d_\e}^{(k+1)\d_\e}\E\,
\le(\E\,\le(I_{2,ij}(r)\le|\mathcal{F}_{k\d_\e}\r)\r.\le|\mathcal{F}_{t_1}\r)\r.\,dr}\\
\vs
\ds{+\int_{t_1}^{([t_1/\d_\e]+1)\d_\e}\E\,\le(I_{2,ij}(r)\le|\mathcal{F}_{t_1}\r)\r.\,dr+
\int_{[t_2/\d_\e]\d_\e}^{t_2}\E\,
\le(\E\,\le(I_{2,ij}(r)\le|\mathcal{F}_{[t_2/\d_\e]\d_\e}\r)\r.\le|\mathcal{F}_{t_1}\r)\r.\,dr.}
\end{array}\]
The distribution of the process
\[v_{1,\e}(r):=\hat{v}_\e(k\d_\e+r),\ \ \ \ \ r
\in\,[0,\d_\e],\]coincides with the distribution of the process
\[v_{2,\e}(r):=\tilde{v}^{u_{\e}(k\d_\e),v_\e(k\d_\e)}(r/\e),\ \ \ \ \ r
\in\,[0,\d_\e],\]where $\tilde{v}^{u_{\e}(k\d_\e),v_\e(k\d_\e)}$
is the solution of problem \eqref{fast} with random frozen slow
component $u_{\e}(k\d_\e)$, random initial datum $v_\e(k\d_\e)$
and noise $\tilde{w}^{Q_2}$ independent of $u_{\e}(k\d_\e)$ and
$v_{\e}(k\d_\e)$. Then, if we set
\[h(x):= D^2_{ij} f (\le<x,P_N
a_1\r>_H,\ldots,\le<x,P_N a_k\r>_H),\ \ \ \ x \in\,H,\] for any
$k=[t_1/\d_\e]+1,\ldots,[t_2/\d_\e]-1$ we have
\[\begin{array}{l}
\ds{\int_{k \d_\e}^{(k+1)
\d_\e}\E\,\le(I_{2,ij}(r)\le|\mathcal{F}_{k\d_\e}\r)\r.\,dr}\\
\vs
\ds{=\int_{0}^{\d_\e}\E\,\le(h(u_{\e}(k\d_\e))\r.\le[\le<G_1(u_{\e}(k\d_\e),v_{1,\e}(r))
Q_{1,N}a_i,G_1(u_{\e}(k\d_\e),v_{1,\e}(r))
Q_{1,N}a_j\r>_H\r.}\\
\vs \ds{\le.\le.-\le<\bar{G}(u_{\e}(k\d_\e))
Q_{1,N}a_i,\bar{G}(u_{\e}(k\d_\e))
Q_{1,N}a_j\r>_H\r]\le|\mathcal{F}_{k\d_\e}\r)\r.\,dr}\\
\vs
\ds{=\int_{0}^{\d_\e}\E\,\le(h(u_{\e}(k\d_\e))\r.\le[\le<G_1(u_{\e}(k\d_\e),v_{2,\e}(r))
Q_{1,N}a_i,G_1(u_{\e}(k\d_\e),v_{2,\e}(r))
Q_{1,N}a_j\r>_H\r.}\\
\vs \ds{\le.\le.-\le<\bar{G}(u_{\e}(k\d_\e))
Q_{1,N}a_i,\bar{G}(u_{\e}(k\d_\e))
Q_{1,N}a_j\r>_H\r]\le|\mathcal{F}_{k\d_\e}\r)\r.\,dr}
\end{array}\]
and, with a change of variables,
\[\begin{array}{l}
\ds{\int_{k \d_\e}^{(k+1)
\d_\e}\E\,\le(I_{2,ij}(r)\le|\mathcal{F}_{k\d_\e}\r)\r.\,dr
=\e\int_{0}^{\frac{\d_\e}{\e}}\E\,\le(h(u_{\e}(k\d_\e))\r.}\\
\vs \ds{\le[\le<G_1(u_{\e}(k\d_\e),
\tilde{v}^{u_{\e}(k\d_\e),v_\e(k\d_\e)}(r))
Q_{1,N}a_i,G_1(u_{\e}(k\d_\e),\tilde{v}^{u_{\e}(k\d_\e),v_\e(k\d_\e)}(r))
Q_{1,N}a_j\r>_H\r.}\\
\vs
\ds{\le.\le.-\le<\bar{G}(u_{\e}(k\d_\e))
Q_{1,N}a_i,\bar{G}(u_{\e}(k\d_\e))
Q_{1,N}a_j\r>_H\r]\le|\mathcal{F}_{k\d_\e}\r)\r.\,dr.}
\end{array}\]
Therefore, due to the Markov property, we obtain
\[\begin{array}{l}
\ds{\int_{k\d_\e}^{(k+1)\d_\e}\E\,\le(I_{2,ij}(r)\le|\mathcal{F}_{k\d_\e}\r)\r.\,dr}\\
\vs \ds{=\e\int_{0}^{\frac{\d_\e}{\e}}\le( \E\,\le<G_1(x,
v^{x,y}(r)) Q_{1,N}a_i,G_1(x,v^{x,y}(r))
Q_{1,N}a_j\r>_Hh(x)\r.}\\
\vs\ds{\le.-\le<\bar{G}(x) Q_{1,N}a_i,\bar{G}(x)
Q_{1,N}a_j\r>_Hh(x)\r)_{\le|x=u_{\e}(k\d_\e),y=v_\e(k\d_\e)\r.}\,dr,}
\end{array}\]
and hence, according to \eqref{sola8},
\[\begin{array}{l}
\ds{\le|\int_{k\d_\e}^{(k+1)\d_\e}\E\,\le(I_{2,ij}(r)\le|\mathcal{F}_{k\d_\e}\r)\r.\,dr\r|\leq
c_{ij}\,\d_\e\a(\d_\e/\e)\le(1+|u_{\e}(k\d_\e)|^2_H+|v_\e(k\d_\e)|^2_H\r),\
\ \ \ \mathbb{P}-\text{a.s}.}
\end{array}\]
Analogously,
\[\begin{array}{l}
\ds{\le|\int_{t_1}^{([t_1/\d_\e]+1)\d_\e}\E\,\le(I_{2,ij}(r)\le|\mathcal{F}_{t_1}\r)\r.\,dr\r|}\\
\vs \ds{\leq c_{ij}\,\d_\e
\le(1-\le\{t_1/\d_\e\r\}\r)\a\le(\le(1-\le\{t_1/\d_\e\r\}\r)\d_\e/\e\r)
\le(1+|u_{\e}(t_1)|^2_H+|v_\e(t_1)|^2_H\r),\ \ \ \
\mathbb{P}-\text{a.s}.} \end{array}\] and
\[\begin{array}{l}
\ds{\le|\int_{[t_2/\d_\e]\d_\e}^{t_2}\E\,\le(I_{2,ij}(r)\le|\mathcal{F}_{[t_2/\d_\e]\d_\e}\r)\r.\,dr\r|}\\
\vs \ds{ \leq c_{ij}\,\d_\e
\le\{t_2/\d_\e\r\}\a\le(\le\{t_2/\d_\e\r\}\d_\e/\e\r)
\le(1+|u_{\e}([t_2/\d_\e]\d_\e)|^2_H+|v_\e([t_2/\d_\e]\d_\e)|^2_H\r),\
\ \ \ \mathbb{P}-\text{a.s}.} \end{array}\]

 Thanks to \eqref{unifue} and \eqref{unifvebis}, the three
 inequality above
imply
\[\lim_{\e\to 0}\E\le|\int_{t_1}^{t_2}\E\le(I^{\e}_{2,ij}(r)\le|\mathcal{F}_{t_1}\r.\r)\,dr\r|=0,\] so that from
\eqref{i13} we conclude that \eqref{sw1} holds.

In an analogous way (just by replacing assumption \eqref{sola8}
with assumption \eqref{sola7}), we can prove that \eqref{sw2}
holds and then, combining together \eqref{sw1} with \eqref{sw2},
we obtain \eqref{sw3}.

\end{proof}

\section{The averaging limit}
\label{sec6}

Before concluding with the proof of the averaging limit, we introduce an approximating slow motion equation and prove a limiting result.

For any $n \in\,\nat$, we define
\[A_{1,n}:=A_1 P_n,\ \ \ \ \ \ Q_{1,n}:=Q_1 P_n,\]
where $P_n$ is the projection of $H$ into
$\text{spam}\le<e_{1,1},\ldots,e_{1,n}\r>$, and we denote by
$u_{\e,n}$ the solution of the problem
\begin{equation}
\label{appro} du(t)=\le[A_{1,n}
u(t)+B_1(u(t),v_\e(t))\r]\,dt+G_1(u(t),v_\e(t))dw^{Q_{1,n}}(t),\ \
\ \ \ u(0)=x.
\end{equation}
Notice that, as $A_{1,n} \in\,\mathcal{L}(H)$ and $Q_{1,n}$ has
finite rank, $u_{\e,n}$ is a strong solution to \eqref{appro},
that is
\begin{equation}
\label{strong}
u_{\e,n}(t)=x+\int_0^t
\le[A_{1,n} u_{\e,n}(s)+B_1(u_{\e,n}(s),v_\e(s))\r]\,ds+\int_0^t
G_1(u_{\e,n}(s),v_\e(s))dw^{Q_{1,n}}(s).
\end{equation}
By standard arguments it is possible to show that for any $p\geq
1$ and $\e>0$
\begin{equation}
\label{limappro} \lim_{n\to \infty}\E \sup_{t
\in\,[0,T]}\,|u_{\e}(t)-u_{\e,n}(t)|_H^2=0.
\end{equation}
Moreover, for any $p\geq 1$ and $\e>0$ it holds
\begin{equation}
\label{unifn} \sup_{n \in\,\nat}\,\E \sup_{t
\in\,[0,T]}\,|u_{\e,n}(t)|_H^p<\infty.
\end{equation}

In analogy to \eqref{sl}, we introduce the Kolmogorov operator associated with
the approximating slow motion equation \eqref{appro}, with frozen
fast component $y \in\,H$, by setting
\[\begin{array}{l}
\ds{\mathcal{L}^n_{sl}\,\varphi(x,y)}\\
\vs \ds{=\frac 12
\text{Tr}\,\le[Q_{1,n}G_1(x,y)D^2\varphi(x)G_1(x,y)Q_{1,n}\r]+\le<A_{1,n}
D\varphi(x),x\r>_H+\le<D\varphi(x),B_1(x,y)\r>_H}\\
\vs \ds{=\frac 12\sum_{i,j=1}^kD^2_{ij} f(\le<x,P_N a_1\r>_H,\ldots,\le<x,P_N
a_k\r>_H)\le<G_1(x,y) Q_{1,N\wedge n}a_i,G_1(x,y) Q_{1,N\wedge n}a_j\r>_H}\\
\vs \ds{+\sum_{i=1}^kD_i f(\le<x,P_N
a_1\r>_H,\ldots,\le<x,P_N a_k\r>_H)\le(\le<x,A_{1,N\wedge n}
a_i\r>_H+\le<B_1(x,y),P_N a_i\r>_H\r).}
\end{array}\]

In the next lemma we show that the Kolmogorov operator
$\mathcal{L}^n_{sl}$ approximates in a proper way the Kolmogorov
operator $\mathcal{L}_{sl}$.

\begin{Lemma}
Assume Hypotheses \ref{H1} and \ref{H2}. Then for any $\varphi
\in\,\mathcal{R}(H)$ and $\e>0$
\begin{equation}
\label{limkol} \lim_{n\to \infty}\E\sup_{t
\in\,[0,T]}\le|\mathcal{L}^n_{\text{sl}}\,\varphi(u_{\e,n}(t),v_\e(t))-
\mathcal{L}_{\text{sl}}\,\varphi(u_\e(t),v_\e(t))\r|=0.
\end{equation}
\end{Lemma}

\begin{proof}
Let
\[\varphi(x)=f(\le<x,P_N a_1\r>_H,\ldots,\le<x,P_N a_k\r>_H),\ \ \ \ \ x \in\,H,\]
for some $k,N \in\,\nat$, $a_1,\ldots,a_k \in\,H$ and $f
\in\,C^\infty_c(\mathbb{R}^k)$. If $n\geq N$, then
\[\mathcal{L}^n_{sl}\,\varphi(x,y)-\mathcal{L}_{sl}\,\varphi(x,y)=0,\ \ \ \ \ x,y \in\,H,\]
 so that for any $\e>0$ and $n\geq N$
\[\mathcal{L}^n_{sl}\,\varphi(u_{\e,n}(t),v_\e(t))-
\mathcal{L}_{sl}\,\varphi(u_\e(t),v_\e(t))=\mathcal{L}_{sl}\,\varphi(u_{\e,n}(t),v_\e(t))-
\mathcal{L}_{sl}\,\varphi(u_\e(t),v_\e(t)).\] Now, due to the
assumptions on the coefficients $B_1$ and $G_1$ and on the funtion
$f$, it is immediate to check that for $x_1,x_2,y \in\,H$
\[|\mathcal{L}_{sl}\,\varphi(x_1,y)-\mathcal{L}_{sl}\,\varphi(x_2,y)|\leq
c\,|x_1-x_2|_H\,(1+|x_1|^2_H+|x_2|_H^2+|y|_H^2).\] Then
\[\begin{array}{l}
\ds{\sup_{t
\in\,[0,T]}\le|\mathcal{L}_{sl}\,\varphi(u_{\e,n}(t),v_\e(t))-
\mathcal{L}_{sl}\,\varphi(u_\e(t),v_\e(t))\r|}\\
\vs \ds{\leq c\,\sup_{t \in\,[0,T]}\,|u_{\e,n}(t)-u_\e(t)|_H
\,(1+\sup_{t \in\,[0,T]}\,|u_{\e,n}(t)|^2_H+\sup_{t
\in\,[0,T]}\,|u_\e(t)|_H^2+\sup_{t \in\,[0,T]}\,|v_\e(t)|_H^2).}
\end{array}\]
According to \eqref{limappro} and \eqref{unifn}, this implies
\eqref{limkol}.
\end{proof}

Finally, we conclude with the proof of the averaging limit. 

\begin{Theorem}
\label{theo5.2}
 Assume Hypotheses \ref{H1}-\ref{H5} and fix any $x
\in\,D((-A_1)^\a)$, with $\a>0$, and any $y \in\,H$. Then, if
$\bar{u}$ is the solution of the averaged equation
\eqref{averaged}, for any $T>0$ and $\eta>0$ we have
\[\text{w}-\lim_{\e\to 0}\,{\cal L}(u_\e)={\cal L}(\bar{u}),\ \ \ \ \ \ \text{in}\ \ C([0,T];H).\]
\end{Theorem}

\begin{proof}
 As $u_{\e,n}$ verifies
\eqref{strong}, for any $\varphi \in\,\mathcal{R}(H)$ we can apply
It\^o's formula to $\varphi(u_{\e,n})$ and we obtain that the
process
\[t \in\,[0,T]\mapsto \varphi(u_{\e,n}(t))-\varphi(x)-\int_0^t\mathcal{L}^n_{sl}\,\varphi(u_{\e,n}(s),v_\e(s))\,ds,\]
is a martingale with respect to $\{\mathcal{F}_t\}_{t
\in\,[0,T]}$. Then, by taking the limit  as $n$ goes to
infinity, due to \eqref{limappro} and to \eqref{limkol}
that for any $\e>0$ we have  the process
\[t \in\,[0,T]\mapsto \varphi(u_{\e}(t))-\varphi(x)-\int_0^t\mathcal{L}_{sl}\,\varphi(u_{\e}(s),v_\e(s))\,ds,\]
is an $\mathcal{F}_t$-martingale. In particular, for any $0\leq
s\leq t\leq T$ and any bounded $\mathcal{F}_s$-measurable random
variable $\Psi$
\begin{equation}
\label{marti} \E\le(\Psi\le[\varphi(u_{\e}(t))-\varphi(u_{\e}(s))-
\int_s^t\mathcal{L}_{sl}\,\varphi(u_{\e}(r),v_\e(r))\,dr\r]\r)=0.\end{equation}

\medskip

Due to the tightness of the sequence $\{\mathcal{L}(u_\e)\}_{\e
\in\,(0,1]}$ in $\mathcal{P}(C_x([0,T];H),\mathcal{E})$ (see
Proposition \ref{corollary1}), there exists a sequence
$\{\e_k\}_{k \in\,\nat}\downarrow 0$ such that the sequence
$\{\mathcal{L}(u_{\e_k})\}_{k \in\,\nat}$ converges weakly to some
$\mathbb{Q}$. If we are able  to identify $\mathbb{Q}$ with
$\mathcal{L}(\bar{u})$, where $\bar{u}$ is the unique mild
solution of the averaged equation \eqref{averaged}, then we conclude that the whole sequence $\{{\cal L}(u_\e)\}_{\e \in\,(0,1]}$ weakly converges to ${\cal L}(\bar{u})$ in $C([0,T];H)$. 

\medskip

We denote by $\E^{\mathbb{Q}}$ and $\E^{\mathbb{Q}_k}$ the
expectations in $(C_x([0,T];H),\mathcal{E})$ with respect to the
probability measures $\mathbb{Q}$ and $\mathbb{Q}_k$,
 where $\mathbb{Q}_k=\mathcal{L}(u_{\e_k})$, and
we denote by $\eta(t)$ the canonical process in
$(C_x([0,T];H),\mathcal{E})$. Then,  for any bounded
$\mathcal{E}_s$-measurable random variable 
\[\Phi=F(\eta(t_1),\ldots,\eta(t_N)), \]
with $F \in\,C_b(\reals^N)$ and $0\leq t_1<\ldots<t_N$, any function
$\varphi \in\,\mathcal{R}(H)$ and any $0\leq s\leq t\leq T$ we
have
\[\begin{array}{l}
\ds{\E^\mathbb{Q}\le(\Phi\le[\varphi(\eta(t))-\varphi(\eta(s))-
\int_s^t\mathcal{L}_{av}\,\varphi(\eta(r))\,dr\r]\r)}\\
\vs \ds{=\lim_{k\to
\infty}\E^{\mathbb{Q}_k}\le(\Phi\le[\varphi(\eta(t))-\varphi(\eta(s))-
\int_s^t\mathcal{L}_{av}\,\varphi(\eta(r))\,dr\r]\r)}\\
\vs \ds{=\lim_{k\to \infty}\E\le(\Phi\circ
u_{\e_k}\le[\varphi(u_{\e_{k}}(t))-\varphi(u_{\e_{k}}(s))-
\int_s^t\mathcal{L}_{av}\,\varphi(u_{\e_{k}}(r))\,dr\r]\r).}
\end{array}\]
In view of \eqref{marti}, this implies
\[\begin{array}{l}
\ds{\E^\mathbb{Q}\le(\Phi\le[\varphi(\eta(t))-\varphi(\eta(s))-
\int_s^t\mathcal{L}_{av}\,\varphi(\eta(r))\,dr\r]\r)}\\
\vs \ds{=\lim_{k\to \infty}\E\le(\Phi\circ u_{\e_k}
\int_s^t\le[\mathcal{L}_{sl}\,\varphi(u_{\e_{k}}(r),v_{\e_{k}}(r))-
\mathcal{L}_{av}\,\varphi(u_{\e_{k}}(r))\r]\,dr\r).}
\end{array}\]
We have
\[\begin{array}{l}
\ds{\le|\E\le(\Phi\circ u_{\e_k}
\int_s^t\le[\mathcal{L}_{sl}\,\varphi(u_{\e_{k}}(r),v_{\e_{k}}(r))-
\mathcal{L}_{av}\,\varphi(u_{\e_{k}}(r))\r]\,dr\r)\r|}\\
\vs \ds{=\le|\E\le(\Phi\circ
u_{\e_k}\int_s^t\E\le(\mathcal{L}_{sl}\,\varphi(u_{\e_{k}}(r),v_{\e_{k}}(r))-
\mathcal{L}_{av}\,\varphi(u_{\e_{k}}(r))\le|\mathcal{F}_s\r.\r)\,dr\r)\r|}\\
\vs \ds{\leq \|\Phi\|_\infty
\E\le|\int_s^t\E\le(\mathcal{L}_{sl}\,\varphi(u_{\e_{k}}(r),v_{\e_{k}}(r))-
\mathcal{L}_{av}\,\varphi(u_{\e_{k}}(r))\le|\mathcal{F}_s\r.\r)\,dr\r|.}
\end{array}\]
Hence, according to \eqref{sw3} we can   conclude  that
\[\E^\mathbb{Q}\le(\Phi\le[\varphi(\eta(t))-\varphi(\eta(s))-
\int_s^t\mathcal{L}_{av}\,\varphi(\eta(r)))\,dr\r]\r)=0.\] This
means that $\mathbb{Q}$ solves  the martingale problem with
parameters $(x,A_1,\bar{B},\bar{G},Q_1)$, and, due to what we have
see in subsection \ref{4.2}, $\mathbb{Q}=\mathcal{L}(\bar{u})$.

\end{proof}

\subsection{Averaging limit  in probability}

In the case the diffusion coefficient $g_1$ in the slow motion equation does not depend on the fast variable, it is possible to prove that the sequence $\{u_\e\}_{\e \in\,(0,1]}$ converges in probability to $\bar{u}$ and not just in weak sense. 

To this purpose we need to replace Hypothesis \ref{H4} with the following stronger condition.

\begin{Hypothesis}
\label{H7} There exists a mapping $\bar{B}_1:H\to H$ such that for
any $T>0$, $t\geq 0$ and $x,y,h \in\,H$
\begin{equation}
\label{sola7bis} \E\,\le|\frac
1T\int_t^{t+T}\le<B_1(x,v^{x,y}(s)),h\r>_H\,ds-\le<\bar{B}_1(x),h\r>_H\r|\leq
\a(T)\,\le(1+|x|_H+|y|_H\r)\,|h|_H,\end{equation}
 for some $\a(T)$ such that
\[\lim_{T\to \infty}\a(T)=0.\]

\end{Hypothesis}

In Subsection \ref{subsec21}, by refeering  our previous paper \cite{cerrai2} we have seen that if the dissipativity constant of the operator $A_1$ is large enough and/or the Lipschitz constants $L_{b_2}$ and $L_{g_2}$ and the constants $\zeta_2$ and $\kappa_2$ introduced in Hypothesis \ref{H1} are  small enough (in this spirit see condition \eqref{condition}), then the fast transition semigroup admits a unique invariant measure $\mu^x$, which is strongly mixing and such that \eqref{sola4} holds. 

In Lemma \ref{conv} we have seen that this implies that for any $\varphi \in\,\text{Lip} (H)$, $T>0$, $x, y \in\,H$ and $t\geq 0$
\[\E\,\le|\frac
1T\int_t^{t+T}\varphi(v^{x,y}(s))\,ds-\int_H
\varphi(z)\,\mu^x(dz)\r|\leq
\frac{c}{\sqrt{T}}\,\le([\varphi]_{\text{
Lip}(H)}(1+|x|_H+|y|_H)+|\varphi(0)|\r).\]
Then, if we apply the inequality above to 
to $\varphi=\le<B_1(x,\cdot),h\r>_H$ and if we set 
$\bar{B}_1(x)=\le<B_1(x,\cdot),\mu^x\r>$, we have that Hypothesis \ref{H7} holds.

\medskip

As $u_\e$ is the mild solution of the slow motion equation in
system \eqref{eq0} (see also \eqref{astratto} for its abstract
version), for any $h \in\,D(A_1)\cap L^\infty(D)$ we have
\begin{equation}
\label{byparts}
\begin{array}{l}
\ds{\le<u_\e(t),h\r>_H=\le<x,h\r>_H+\int_0^t\le<u_{\e}(s),A_1
h\r>_H\,ds+\int_0^t
\le<\bar{B}_1(u_{\e}(s)),h\r>_H\,ds}\\
\vs \ds{+\int_0^t \le<G_1(u_\e(s)) h,dw^{Q_1}(s)\r>_H+R_\e(t),}
\end{array}\end{equation}
where
\[R_\e(t):=\int_0^t
\le<B_1(u_{\e}(s),v_\e(s))-\bar{B}_1(u_{\e}(s)),h\r>_H\,ds.\] In
order to prove the averaging limit, we need the following key
lemma, which is the counterpart of Lemma \ref{lemma51}.

\begin{Lemma}
\label{re} Assume Hypotheses \ref{H1}, \ref{H2} and \ref{H7} and fix $T>0$.
Then, for any $x \in\,D((-A_1)^\a)$, with $\a>0$, and $y, h
\in\,H$
 we have
\[\lim_{\e\to 0}\,\E\sup_{t \in\,[0,T]}|R_\e(t)|=0.\]
\end{Lemma}

\begin{proof}
{\em Step 1.}
We prove that 
\begin{equation}
\label{sola60} \lim_{\e\to 0}\,\E \sup_{t \in\,[0,T]}\le|\int_0^t
\le<B_1(u_\e([s/\d_\e]\d_\e),\hat{v}_\e(s))-\bar{B}_1(u_\e(s)),h\r>_H\,ds\r|=0,
\end{equation}
where $\hat{v}_\e(t)$ is the solution of problem    \eqref{hatve}  
Let $k=0,\ldots,[T/\d_\e]$ be fixed. If we take a noise
$\tilde{w}^{Q_2}(t)$ independent of $u_\e(k \d_\e)$ and $v_\e(k
\d_\e)$ in the fast motion equation \eqref{fast}, it is immediate
to check that  the process \[z_{1,\e}(s):=\tilde{v}^{u_\e(k
\d_\e),v_\e(k \d_\e)}(s/\e),\ \ \ \ s \in\,[0,\d_\e],\] coincides
in distribution with the process
\[z_{2,\e}(s):=\hat{v}_\e(k\d_\e+s),\ \ \ \ s \in\,[0,\d_\e].\]
 This means that
\[\begin{array}{l}
\ds{\E\le|\int_{k\d_\e}^{(k+1)\d_\e}\le<B_1(u_\e(k
\d_\e),\hat{v}_\e(s))-\bar{B}_1(u_\e(k
\d_\e)),h\r>_H\,ds\r|}\\
\vs \ds{=\E\le|\int_{0}^{\d_\e}\le<B_1(u_\e(k
\d_\e),z_{2,\e}(s))-\bar{B}_1(u_\e(k \d_\e)),h\r>_H\,ds\r|}\\
\vs \ds{={\E}\le|\int_{0}^{\d_\e}\le<B_1(u_\e(k
\d_\e),z_{1,\e}(s))-\bar{B}_1(u_\e(k \d_\e)),h\r>_H\,ds\r|}\\
\vs \ds{=\d_\e {\E}\le|\frac
1{\zeta_\e}\int_{0}^{\zeta_\e}\le<B_1(u_\e(k
\d_\e),\tilde{v}^{u_\e(k \d_\e),v_\e(k
\d_\e)}(s))-\bar{B}_1(u_\e(k \d_\e)),h\r>_H\,ds\r|.}
\end{array}\]
Hence, according to Hypothesis \ref{H7}, due to \eqref{unifue} and
\eqref{unifvebis}, we have
\[\begin{array}{l}
\ds{\E\,\le|\int_{k\d_\e}^{(k+1)\d_\e}\le<B_1(u_\e(k
\d_\e),\hat{v}_\e(s))-\bar{B}_1(u_\e(k \d_\e)),h\r>_H\,ds\r|}\\
\vs \ds{ \leq \d_\e\,\a(\zeta_\e)\le(1+{\E}\,|u_\e(k
\d_\e)|_H+{\E}\,|v_\e(k \d_\e)|_H\r)\,|h|_H\leq
c_T\le(1+|x|_H+|y|_H\r)|h|_H\,\d_\e\,\a(\zeta_\e).}
\end{array}\]
This allows to obtain \eqref{sola60}. Actually, we have
\[\begin{array}{l}
\ds{\E \sup_{t \in\,[0,T]}\le|\int_0^t
\le<B_1(u_\e([s/\d_\e]\d_\e),\hat{v}_\e(s))-\bar{B}_1(u_\e(s)),h\r>_H\,ds\r|}\\
\vs \ds{\leq \sum_{k=0}^{\le[T/\d_\e\r]} \E \le|\int_{k
\d_\e}^{(k+1) \d_\e}
\le<B_1(u_\e([s/\d_\e]\d_\e),\hat{v}_\e(s))-\bar{B}_1(u_\e(k \d_\e)),h\r>_H\,ds\r|}\\
\vs \ds{+\sum_{k=0}^{\le[T/{\d_\e}\r]}\int_{k \d_\e}^{(k+1) \d_\e}
\E \,\le|\le<\bar{B}_1(u_\e(k
\d_\e))-\bar{B}_1(u_\e(s)),h\r>_H\r|\,ds,}
\end{array}\]
and, as $\bar{B}_1$ is Lipschitz continuous, thanks to
\eqref{sola30} we get
\[\begin{array}{l}
\ds{\E \sup_{t \in\,[0,T]}\le|\int_0^t
\le<B_1(u_\e([s/\d_\e]\d_\e),\hat{v}_\e(s))-\bar{B}_1(u_\e(s)),h\r>_H\,ds\r|}\\
\vs \ds{\leq c_T\le[T/\d_\e\r]\le(1+|x|_H+|y|_H\r)\d_\e
\a(\zeta_\e) +c_T \le[T/\d_\e\r]\le(1+|x|_\a+|y|_H\r)
|h|_H\,\d_\e^{1+\beta(\a)},}
\end{array}\]
and \eqref{sola60} follows.\medskip

{\em Step 2.} It holds
\[\lim_{\e\to 0}\,\sup_{t \in\,[0,T]}\E\,|R_\e(t)|=0.\]
Thanks to \eqref{sola60}, we have
\[\begin{array}{l}
\ds{\limsup_{\e\to 0}\,\E \sup_{t \in\,[0,T]}\le|\int_0^t
\le<B_1(u_{\e}(s),v_\e(s))-\bar{B}_1(u_{\e}(s)),h\r>_H\,ds\r|}\\
\vs \ds{\leq \limsup_{\e\to 0}\,\E\,\int_0^T
\le|\le<B_1(u_{\e}(s),v_\e(s))-B_1(u_{\e}([s/\d_\e]\d_\e),\hat{v}_\e(s)),h\r>_H\r|\,ds}\\
\vs \ds{ +\lim_{\e\to 0}\,\E \sup_{t \in\,[0,T]}\le|\int_0^t
\le<B_1(u_{\e}([s/\d_\e]\d_\e),\hat{v}_\e(s))-\bar{B}_1(u_{\e}(s)),h\r>_H\,ds\r|}\\
\vs \ds{=\limsup_{\e\to 0}\,\E\,\int_0^T
\le|\le<B_1(u_{\e}(s),v_\e(s))-B_1(u_{\e}([s/\d_\e]\d_\e),\hat{v}_\e(s)),h\r>_H\r|\,ds.}
\end{array}\]
By using \eqref{sola30} we have
\[\begin{array}{l}
\ds{\E\,\int_0^T
\le|\le<B_1(u_{\e}(s),v_\e(s))-B_1(u_{\e}([s/\d_\e]\d_\e),\hat{v}_\e(s)),h\r>_H\r|\,ds}\\
\vs \ds{\leq c
\int_0^T\le(\,\E\,|u_{\e}(s)-u_{\e}([s/\d_\e]\d_\e)|_H+\E\,|v_\e(s)-\hat{v}_\e(s)|_H\r)\,ds\,|h|_H}\\
\vs \ds{\leq
c_T(1+|x|_\a+|y|_H)|h|_H\,\d_\e^{\beta(\a)}+T\,\sup_{t
\in\,[0,T]}\E\,|v_\e(t)-\hat{v}_\e(t)|_H\,|h|_H,}
\end{array}\]
and then, due to \eqref{sola52}, we have
\[\lim_{\e\to 0} \E\,\int_0^T
\le|\le<B_1(u_{\e}(s),v_\e(s))-B_1(u_{\e}([s/\d_\e]\d_\e),\hat{v}_\e(s)),h\r>_H\r|\,ds=0.\]
This allows to conclude.
\end{proof}

\begin{Theorem}
\label{averaging} Assume that
the diffusion coefficient $g_1$ in the slow motion equation does not depend on the fast variable $v_\e$  and fix $x \in\,D((-A_1)^\a)$, for some
$\a>0$, and  $y \in\,H$. Then, under Hypotheses \ref{H1}, \ref{H2} and \ref{H7}
for any $T>0$ and $\eta>0$ we have
\begin{equation}
\label{wit30} \lim_{\e\to
0}\mathbb{P}\,\le(|u_\e-\bar{u}|_{C([0,T];H)}>\eta\,\r)=0,
\end{equation}
where $\bar{u}$ is the solution of the averaged equation
\eqref{averaged}.
\end{Theorem}

\begin{proof}
We have seen that the family of probability
measures $\{\mathcal{L}(u_\e)\}_{\e \in\,(0,1]}$ is tight in
$\mathcal{P}(C([0,T];H))$, for any fixed $T>0$. Moreover, due to
Lemma \ref{re} the remainder $R_\e(t)$ converges to zero in
$L^1(\Omega,\mathcal{F},\mathbb{P})$, uniformly with respect to $t
\in\,[0,T]$. Then the proof of \eqref{wit30} proceeds as
in \cite[Theorem 5.4]{cf}. 

Since the sequence $\{\mathcal{L}(u_\e)\}_{\e>0}$ is tight in
$\mathcal{P}(C([0,T];H))$, if we fix any two sequences
$\{\e_n\}_{n \in\,\nat}$ and $\{\e_m\}_{m \in\,\nat}$ which
converge to zero, due to the Skorokhod theorem we can find
subsequences $\{\e_{n(k)}\}_{k \in\,\nat}$ and $\{\e_{m(k)}\}_{k
\in\,\nat}$ and a sequence
\[\{X_k\}_{k \in\,\nat}:=\le\{(u_1^k,u_2^k)\r\}_{k \in\,\nat}\subset
\mathcal{C}:=C([0,T];H)\times
C([0,T];H),\] defined on some probability
space $(\hat{\Omega},\hat{\F},\hat{\Pro})$, such that
\begin{equation}
\label{sola100}
\mathcal{L}(X_k)=\mathcal{L}((u_{\e_{n(k)}},u_{\e_{m(k)}})), \ \
\ \ \ k \in\,\nat, \end{equation} and $X_k$ converges  to some
$X:=(u_1,u_2) \in\,\mathcal{C}$, $\hat{\Pro}$-a.s. We will show that if  $u_1=u_2$, then  there exists some $u
\in\,C([0,T];H)$ such that the whole sequence $\{u_\e\}_{\e>0}$
converges to $u$ in probability.

For $k \in\,\nat$ and $i=1,2$, we define
\begin{equation}
\label{sola70} \begin{array}{l} \ds{R^k_i(t):=\langle
u^k_i(t),h\rangle_H-\langle x,h\rangle_H-\int_0^t\langle
u^k_i(s),A_1 h\rangle_H}\\
\vs \ds{-\int_0^t\langle
\bar{B}_1(u^k_i(s)),h\rangle_H-\int_0^t\langle
G_1(u^k_1(s))h,d\hat{w}^{Q_1}(s)\rangle_H,}
\end{array}
\end{equation}
where
\[\hat{w}(t,\xi)=\sum_{j \in\,\nat}Q_1 e_j(\xi) \hat{\beta}_j(t),\]
and $\{\hat{\beta}_j(t)\}_{j \in\,\nat }$ is  a sequence of mutually independent standard Brownian motions on $(\hat{\Omega},\hat{\F},\hat{\Pro})$.
In view of \eqref{sola100},  by using Lemma \ref{re}, we have
\[\lim_{k\to \infty}\hat{\E} \sup_{t \in\,[0,T]}\,|R^k_i(t)|=0,\ \ \ \ i=1,2,\]
then, if we pass possibly to a subsequence, we can take the
$\hat{\mathbb{P}}$-almost sure  limit in \eqref{sola70}, and we
get that  both $u_1$ and $u_2$ solve the problem
\[\begin{array}{l}
\ds{ \le<u(t),h\r>_H=\le<x,h\r>_H+\int_0^t \le<u(s),A_1
h\r>_H\,ds}\\
\vs
\ds{+\int_0^t\le<\bar{B}_1(u(s)),h\r>_H\,ds+\int_0^t\le<G_1(u(s))h,d\hat{w}^{Q_1}(s)\r>_H,}
\end{array}\] for any $h \in\,D(A_1)\cap L^\infty(D)$. This means
that $u_1=u_2$, as they coincide with the unique solution of
equation
\[du(t)=\le[A_1 u(t)+\bar{B}_1(u(t))\r]\,dt+G_1(u(t))\,d\hat{w}^{Q_1}(t),\ \ \ \ \ u(0)=x,\]
 and then the sequences $\{{\cal L}(u_{\e_{n(k)}})\}$ and $\{{\cal L}(u_{\e_{m(k)}})\}$ weakly converge to the same limit.

This allows to conclude that \eqref{wit30} is true, as in  Gy\"ongy and Krylov \cite[Lemma 1.1]{gk} it is proved that if $\{Z_n\}$  is a sequence of random element in a Polish space $X$, then $\{Z_n\}$ converges in probability to a $X$-valued random element if and only if for every pair of subsequences $\{Z_l\}$ and $\{Z_m\}$ there exists a subsequence $v_k=(Z_{l_k},Z_{m_{k}})$ converging weakly to a random element $v$ supported on the diagonal of $X\times X$.
\end{proof}

\end{document}